\documentclass{amsart}

\usepackage{amsmath}
\usepackage{amssymb}
\usepackage{wasysym}
\usepackage{graphicx}
\usepackage[all]{xy}
\input{quivers.sty}

\newtheorem{theorem}{Theorem}[section]
\newtheorem{lemma}[theorem]{Lemma}

\newtheorem{corollary}[theorem]{Corollary}
\newtheorem{proposition}[theorem]{Proposition}

\theoremstyle{definition}
\newtheorem{definition}[theorem]{Definition}
\newtheorem{example}[theorem]{Example}

\theoremstyle{remark}
\newtheorem{remark}[theorem]{Remark}

\newtheorem{problem}[theorem]{Problem}
\newtheorem{conjecture}[theorem]{Conjecture}

\numberwithin{equation}{section}

\newcommand{\GL}{\operatorname{GL}}
\newcommand{\Gr}{\operatorname{Gr}}
\newcommand{\module}{\operatorname{mod}}

\newcommand{\Rep}{\operatorname{Rep}}
\newcommand{\Mod}{\operatorname{Mod}}
\newcommand{\Hom}{\operatorname{Hom}}

\newcommand{\End}{\operatorname{End}}
\newcommand{\Ext}{\operatorname{Ext}}
\newcommand{\ext}{\operatorname{ext}}
\newcommand{\E}{\operatorname{E}}
\newcommand{\Ker}{\operatorname{Ker}}

\newcommand{\Img}{\operatorname{Im}}
\newcommand{\codim}{\operatorname{codim}}
\newcommand{\T}{\operatorname{T}}

\renewcommand{\L}{\operatorname{L}}
\newcommand{\R}{\operatorname{R}}
\newcommand{\Perp}{{^\perp}\!}

\newcommand{\rank}{\operatorname{rank}}
\newcommand{\reg}{{\operatorname{reg}}}
\newcommand{\red}{\operatorname{red}}

\newcommand{\SI}{\operatorname{SI}}
\newcommand{\SIR}{\operatorname{SIR}}
\newcommand{\Spec}{\operatorname{Spec}}
\newcommand{\Proj}{\operatorname{Proj}}
\newcommand{\dslash}[1]{/\!\!/_{\!\!_#1}}
\newcommand{\mc}[1]{\mathcal{#1}}
\newcommand{\mb}[1]{\mathbb{#1}}

\newcommand{\br}[1]{\overline{#1}}
\newcommand{\innerprod}[1]{\langle#1\rangle}
\renewcommand{\ss}[2]{^{\sigma_{#1}\cdot{\rm #2}}}
\newcommand{\sm}[1]{\left(\begin{smallmatrix}#1\end{smallmatrix}\right)}

\makeatletter
\def\equalsfill{$\m@th\mathord=\mkern-7mu
\cleaders\hbox{$\!\mathord=\!$}\hfill
\mkern-7mu\mathord=$}
\makeatother

\begin{document}

\title{Moduli of Representations I. Projections from Quivers}

\author{Jiarui Fei}
\address{Department of Mathematics, University of Michigan, Ann Arbor, MI 48109, USA}
\email{jiarui@umich.edu}
\thanks{}

\subjclass[2010]{Primary 16G20; Secondary 14D20, 14L24, 16E30}

\date{}
\dedicatory{}
\keywords{Representation of Algebra, Quiver, Moduli space, Geometric invariant, Semi-invariant Ring, Bimodule, Orthogonal Projection, Tilting, Birational Geometry, Variation of GIT, Fano Variety}

\begin{abstract} We use categorical method and birational geometry to study moduli spaces of quiver representations. From certain ``representable" functor, we construct a birational transformation from the moduli space of representations of one quiver to another new quiver with one vertex less. The dimension vector and the stability for the new moduli are determined functorially. We introduce several relative notions of stability to study such a birational transformation. The essential case is proven to be the usual blow-ups. Moreover, we compare the induced ample divisors for the two moduli. We illustrate this theory by various examples.
\end{abstract}

\maketitle

\section{Introduction}
We propose to use birational geometry in a functorial way to study moduli spaces of representations of algebras in the sense of geometric invariant theory, e.g. King \cite{Ki}. The idea is very simple and can be roughly described as follows. For any algebra $A$, we denote by $\Mod(A)$ the abelian category of finite-dimensional left modules of $A$ and $\Rep_\alpha(A)$ the usual affine variety of $\alpha$-dimensional representations of $A$. To study the moduli space of representations in an irreducible component $\mc{C}$ of $\Rep_\alpha(A)$, we appropriately choose an $A$-$B$-bimodule $T$ and consider the functor $F=\Hom_A(T,-):\Mod(A)\to\Mod(B)$. If $F$ maps a general representation in an irreducible component $\mc{C}$ to a general representation in an irreducible component $\mc{C}_F$. Then $F$ induces a rational map $f$ from $\mc{C}$ to the moduli stacks $\Mod_{\mc{C}_F}(B)$. This procedure is explained in detail in Section \ref{S:bimodule}.
Let $\Mod_{\mc{C}}^{\sigma}(A)$ be the GIT quotient of $\sigma$-semi-stable representations of $\mc{C}$. It is a quotient of an open substack of $\Mod_{\mc{C}}(A)$. Under some nice property of $F$,
$f$ descends to a rational map of the GIT quotients $\Mod_{\mc{C}}^{\sigma}(A)\to \Mod_{\mc{C}_F}^{\sigma_F}(B)$ for certain stability $\sigma_F$. This construction is explained in detail in Section \ref{S:FC}. If we can understand this rational map and the quotient downstairs, then the other one can be understood as well.

Although this idea can be applied to any algebra, as a first attempt we only focus on finite-dimensional quiver algebras. We have the following reasons for this. The varieties of representations are all affine spaces, in particular smooth and irreducible. Otherwise, one has to detect irreducible components and works with possibly highly singular varieties. For smooth variety, Mumford's semi-invariant construction gives all possible good quotients in the category of varieties or schemes. Moreover, the simple nature of quiver algebras allow us to refine many general results. In fact, we think that it is even possible to do some classification in low dimensions. Finally, this theory has a good application in Schubert calculus \cite[Section 7]{DW2}. Since we hope to stay in the world of quivers, the functors that we pick are very special. They are orthogonal projections, introduced by Geigle and Lenzing \cite{GL}. Its induced rational map turns out to be birational in the orthogonal case. This approach may be also suggested by Schofield at the end of the introduction of \cite{S3}. In \cite{S3}, Schofield studied the birational equivalence among the moduli spaces targeting the rationality problem. However, we focus more on the geometry of the birational transformation. We also show by example that the answer to the question proposed in \cite[Introduction]{S3} is negative in general.

We state several main results of this notes. Given a dimension vector $\alpha$ with a stability $\sigma_\beta$, we consider an exceptional representation $E$ (right) orthogonal to a general $\alpha$-dimensional representation. The right orthogonal projection functor $\R_{Q_E}$ projects the module category $\Mod(Q)$ to $\Mod(Q_E)$ for another quiver $Q_E$ with one vertex less. This functor also determine a new dimension vector $\alpha_\epsilon$ and a new stability $\sigma_{\beta_E^\vee}$ of $Q_E$. Let $q: \Rep_\alpha\ss{\beta}{ss}(Q)\to \Mod_\alpha^{\sigma_\beta}(Q)$ be the GIT quotient map.
\begin{theorem} The functor $\R_{Q_E}$ induces a birational transformation $$\varphi_E:\Mod_\alpha^{\sigma_\beta}(Q)\dashrightarrow\Mod_{\alpha_\epsilon}^{\sigma_{\beta_E^\vee}}(Q_E),$$
which maps $q(\Rep_\alpha\ss{\beta}{ss}(\Perp E))$ isomorphically onto its image.
\end{theorem}

We say a stability $\sigma_\beta$ lying on a {\em wall} of $\alpha$ if $\Rep_\alpha(Q)$ contains a strict $\sigma_\beta$-semi-stable point.
The orthogonal projection functor also determines a stability $\sigma_{\tilde{\beta}_E^\vee}$ on $Q$.

\begin{theorem} \label{T:intro2} Assume some mild conditions. If $\beta\tilde{\beta}_E^\vee$ only crosses a single wall, then
$\varphi_E$ is the blow-up of $\Mod_{\alpha_\epsilon}^{\sigma_{\beta_\epsilon^\vee}}(Q_E)$ along an irreducible subvariety which has dimension $-\innerprod{\epsilon,\alpha-\epsilon}_Q$ if non-empty. 
\end{theorem}

Both the blow-up locus and the exceptional divisor can be described in terms of the representation theory. The exceptional divisor is exactly the quotient of those representations having $E$ as a quotient representation, and the blow-up locus is exactly the quotient of image of $\R_{Q_E}$ on those representations having $E$ as a subrepresentation.

The GIT construction also gives us an ample divisor on the quotient for free, namely the one induced from the $G$-linearization. We denote by $D_\beta$ the one from the stability $\sigma_\beta$.

\begin{theorem} Under some mild assumption, $D_\beta=\varphi_E^*(D_{\beta_\epsilon^\vee})-\innerprod{\beta,\epsilon}_Q E_\beta.$\\
Here, $E_\beta$ is an irreducible divisor which can be explicitly described. In the situation of Theorem \ref{T:intro2}, it is the exceptional divisor of the blow-up.
\end{theorem}

This notes is organized as follows. In Section \ref{S:GIT}, we review some basics on the GIT moduli of quiver representations.
In Section \ref{S:bimodule}, we explain how to construct algebraic morphisms out of a functor represented by a bimodule. Lemma \ref{L:morphism} is fundamental to the whole notes.
In Section \ref{S:OP}, we review some basics on the orthogonal projection. A concrete description of the representability of this functor (Lemma \ref{L:TD}) may be new. Next, we make a crucial definition of $E$-regular representations and give several characterizations (Lemma \ref{L:Ereg}). We provide a family of surface examples to illustrate Lemma \ref{L:morphism}.
In Section \ref{S:FR}, we introduce the fundamental rank to prove some basic properties of the exceptional set $\mc{E}_\alpha$ including its irreducibility.
In Section \ref{S:FC}, we introduce a series of relative notions to study how the stability changes under the orthogonal projections. After that, we give our first main result Theorem \ref{T:birational} on a functorial construction of birational transformations of the GIT quotients. We illustrate this result by continuing the examples in Section \ref{S:OP}.
One important corollary builds up the connection between the functorial construction and the wall-crossing, which leads to the important definition of the core of the $G$-ample cone.
The functorial construction is further refined in Section \ref{S:SC} by using a result of M. Thaddeus. We show in Theorem \ref{T:blow-up} that the essential case of that birational transformation is the usual blow-up. This is our second main result. It is furnished with an interesting example in dimension three.
In Section \ref{S:IAD}, we compare the induced ample divisors under the birational transformation. Theorem \ref{T:IAD} is our third main result. We also illustrate it by examples discussed before. We consider in particular the induced divisors from the anti-canonical characters and show they agree with the anti-canonical classes in most cases.

Finally, we want to make several little remarks on our language and notations. Although the moduli spaces we interested in are the classical GIT quotients, for our convenience we introduce the moduli stacks because the functors that we consider work better at the level of moduli stacks. We use $\Rep$ to denote the affine variety of representations and replace $\Rep$ by $\Mod$ to denote the corresponding moduli stack. For example, $\Mod_\alpha(Q)$ is the moduli stack of all $\alpha$-dimensional representations of $Q$, and similarly $\Mod_\alpha\ss{\beta}{ss}(Q)$ is the moduli stack of all $\alpha$-dimensional $\sigma_\beta$-semi-stable representations. Note the difference between the moduli stack $\Mod_\alpha\ss{\beta}{ss}(Q)$ and the GIT quotient $\Mod_\alpha^{\sigma_\beta}(Q)$. When used without any subscript or superscript, it coincides with the module category. We indulge ourself in this abuse of notation.

The second remark is that if we put an irreducible variety in an argument for a single representation, we mean a general representation in that variety. For example, $\hom_Q(M,C^M)$ means the dimension of $\Hom_Q(M,N)$, where $N$ is a general representation of $C^M$. If the variety is $\Rep_\alpha(Q)$, we abbreviate it by $\alpha$, e.g. $\hom_Q(\alpha,C^M)$. If the variety is a subvariety of a product of two varieties, then this variety can replace a bi-argument. For example, $\hom_Q(C)$ means the dimension of $\Hom_Q(M,N)$ for a general element $(M,N)\in C$.

Finally, our vectors are exclusively row vectors. If an arrow of a quiver is denoted by a lowercase letter, then we use the same capital letter for its linear map of a representation. For direct sum of $n$ copy of $M$, we write $nM$ instead of the traditional $M^{\oplus n}$. If $M$ has a projective (resp. injective) presentation $P_1\to P_0$ (resp. $I_0\to I_1$), then we may use the shorthand $M=P_1\to P_0$ (resp. $M=I_0\to I_1$). Similar notation for injective presentation For any other unexplained terminology or notation in the representation theory of quivers, we refer the readers to \cite{DW2}.

\section{Basics on Moduli of Quiver Representations} \label{S:GIT}
We provide some preliminary background on the GIT moduli spaces of quiver representations. For a good introduction or more detailed treatment, we recommend \cite{Ki,R}.
Let $Q$ be a finite quiver with the set of vertices $Q_0$ and the set of arrows $Q_1$. If $a\in Q_1$ is an arrow, then $ta$ and $ha$ denote its tail and its head respectively. Throughout $k$ is a algebraically closed field of characteristic $0$. Fix a {\em dimension vector} $\alpha$, the space of all $\alpha$-dimensional representations is
$$\Rep_\alpha(Q):=\bigoplus_{a\in Q_1}\Hom(k^{\alpha(ta)},k^{\alpha(ha)}).$$
The product of general linear group $\GL_\alpha:=\prod_{v\in Q_0}\GL_{\alpha(v)}$ acts on $\Rep_\alpha(Q)$ by the natural base change. This action has a {\em kernel}, which is the multi-diagonally embedded $k^*$. Two representations $M,N\in\Rep_\alpha(Q)$ are isomorphic if they lie in the same $\GL_\alpha$-orbit. Any {\em character} or {\em weight} of $\GL_\alpha$ has the form
$$\{g(v)\mid v\in Q_0\}\to\prod\big(\det g(v)\big)^{\sigma(v)},\ \sigma(v)\in\mb{Z}.$$
We define the subgroup $\GL_\alpha^\sigma$ to be the kernel of the character map. The semi-invariant ring $\SIR_\alpha^\sigma(Q):=k[\Rep_\alpha(Q)]^{\GL_\alpha^\sigma}$ of weight $\sigma$
is $\sigma$-graded: $\oplus_{n\geqslant 0} \SI_\alpha^{n\sigma}(Q)$, where
$$\SI_\alpha^{\sigma}(Q):=\{f\in k[\Rep_\alpha(Q)]\mid g(f)=\sigma(g)f, \forall g\in\GL_\alpha\}.$$

A representation $M\in\Rep_\alpha(Q)$ is called {\em $\sigma$-semi-stable} if there is some non-constant $f\in \SIR_\alpha^\sigma(Q)$ such that $f(M)\neq 0$. It is called {\em stable} if the orbit $\GL_\alpha^\sigma M$ is closed of dimension $\dim\GL_\alpha^\sigma-1$. We denote the set of all $\sigma$-semi-stable (resp. $\sigma$-stable, $\sigma$-unstable) representations in $\Rep_\alpha(Q)$ by $\Rep_\alpha\ss{}{ss}(Q)$ (resp. $\Rep_\alpha\ss{}{st}(Q)$, $\Rep_\alpha\ss{}{un}(Q)$). When $\Rep_\alpha\ss{}{ss}(Q)$ (resp. $\Rep_\alpha\ss{}{st}(Q)$) is non-empty, we say that $\alpha$ is $\sigma$-semi-stable (resp. $\sigma$-stable). Based on Hilbert-Mumford criterion, King provides a simple criterion for the stability of a representation.
\begin{lemma} \cite[Proposition 3.1]{Ki}  \label{L:King} A representation $M$ is $\sigma$-semi-stable (resp. $\sigma$-stable) if and only if $\sigma(\br{M})=0$ and $\sigma(\br{L})\leqslant 0$ (resp. $<0$) for any non-trivial subrepresentation $L$ of $M$.
\end{lemma}

The GIT quotient with respect to a linearization of character $\sigma$ is
$$\Mod_\alpha^\sigma(Q):=\Proj(\oplus_{n\geqslant 0} \SI_\alpha^{n\sigma}(Q)),$$
which is projective over the ordinary quotient $\Spec(k[\Rep_\alpha(Q)]^{\GL_\alpha})$.
If we assume that $Q$ has no oriented cycles, then $k[\Rep_\alpha(Q)]^{\GL_\alpha}=k$, so the GIT quotient is projective. In this notes, we assume that $Q$ has no oriented cycles, or equivalently the associated path algebra $kQ$ is finite-dimensional. So the category $\Rep(Q)$ has enough projective and injective objects.

The induced quotient map $q:\Rep_\alpha\ss{}{ss}(Q)\to\Mod_\alpha^\sigma(Q)$ is a {\em good categorical quotient} and the restriction of $q$ on $\Rep_\alpha\ss{}{st}(Q)$ is a {\em geometric quotient}. In general, the quotient $\Mod_\alpha^\sigma(Q)$ may be singular, but they are always normal and have {\em rational singularities} because in general these two properties are inherited by the invariant ring.

Before treating the general case, we would like to mention several reduction steps. If $\alpha$ is strictly $\sigma$-semi-stable, then it has a {\em $\sigma$-stable decomposition}: $\alpha=c_1\alpha_1\dot{+}c_2\alpha_2\dot{+}\cdots\dot{+}c_r\alpha_r$ \cite[3.2]{DW2}. \cite[Theorem 3.20]{DW2} says that in this case the GIT quotient $\Mod_\alpha^\sigma(Q)$ is just a symmetric product of $\Mod_{\alpha_i}^\sigma(Q)$'s. So we only need to deal with the more geometric cases when $\alpha$ is $\sigma$-stable. Next, if the quotient is $1$-dimensional, it is a rational normal curve, which must be $\mb{P}^1$. So we will focus on the case when the quotient has dimension at least $2$. The dimension of the quotient is given by
$$\dim\Rep_\alpha(Q)-(\dim\GL_\alpha^\sigma-1)-1=1-\innerprod{\alpha,\alpha}_Q,$$
where $\innerprod{-,-}_Q$ is the {\em Euler form} associated to $Q$. So in this case, $\alpha$ must be a {\em non-isotropic imaginary} {\em Schur root}. Here, non-isotropic imaginary means that $\innerprod{\alpha,\alpha}_Q<0$ and Schur root means that $\Hom_Q(M,M)=k$ for general $M\in\Rep_\alpha(Q)$.

Let $\Sigma_\alpha(Q)$ denote the {\em $G$-ample cone} of the action, that is, the set of all weights $\sigma$ such that $\alpha$ is $\sigma$-semi-stable. It turns out that such a weight $\sigma$ must be of the form $\sigma_\beta:=\innerprod{\beta,-}_Q$ for some dimension vector $\beta$.

For any $N\in\Rep_{\beta}(Q)$, we apply the functor $\Hom_Q(-,M)$ to the {\em canonical resolution} of $N$:
\begin{equation}\label{eq:canproj} 0\xrightarrow{}\bigoplus_{a\in Q_1}P_{ha}\otimes N(ta)\xrightarrow{}\bigoplus_{v\in Q_0}P_v\otimes N(v)\xrightarrow{}N\to 0.\end{equation}
Here, $P_v$ is the indecomposable projective representation with top $S_v$, the one-dimensional simple representation supported on the vertex $v$.
Then we get the following exact sequence
\begin{equation} \label{eq:canseq} \Hom_Q(N,M)\hookrightarrow\bigoplus_{v\in Q_0}\Hom(N(v),M(v))\xrightarrow{\phi_M^N}\bigoplus_{a\in Q_1}\Hom(N(ha),M(ta))\twoheadrightarrow\Ext_Q(N,M).\end{equation}

If $\innerprod{\beta,\alpha}_Q=0$, then $\phi_M^N$ is a square matrix. Following Schofield \cite{S1}, we define $c(N,M):=\det\phi_M^N$. Note that $c(N,M)\neq 0$ if and only if $\Hom_Q(N,M)=\Ext_Q(N,M)=0$. We denote $c_N:=c(N,-)$ and dually $c^M:=c(-,M)$. Note that $c_N\in\SI_\alpha^{\sigma_\beta}(Q)$. We call them Schofield's semi-invariants. In fact, they span $\SI_\alpha^{\sigma_\beta}(Q)$ over $k$ \cite[Theorem 1]{DW1}. So if $M$ is a $\sigma_\beta$-semi-stable representation, then there is some $n\in\mb{N}$ and $N'\in\Rep_{n\beta}(Q)$ such that $c_{N'}(M)\neq 0$.
Clearly $c_N$ defines a $\GL_\alpha$-invariant divisor on $\Rep_\alpha(Q)$ and we denote it by $C_N$. Let $q:\Rep_\alpha\ss{\beta}{ss}(Q)\to\Mod_\alpha^{\sigma_\beta}(Q)$ be the quotient map, then $C_N$ induces an ample effective divisor $q_*(C_N)$ on $\Mod_\alpha^{\sigma_\beta}(Q)$. All divisors of this form are linear equivalent and we denote it by $D_\beta$. So we get not only a projective quotient but also an induced ample divisor for free. Its space of global section $H^0(\Mod_\alpha^{\sigma_\beta}(Q),D_\beta)$ is isomorphic to $\SI_\alpha^{\sigma_\beta}(Q)$. Certainly, there are dual statements for $c^M$.

For any dimension vector $\alpha$, there is a canonical choice of weight, which is $$\sigma_{ac}=\innerprod{\alpha,-}_Q-\innerprod{-,\alpha}_Q=\innerprod{\alpha+\tau^{-1}\alpha,-}_Q=-\innerprod{-,\alpha+\tau\alpha}_Q.$$
To simplify our notation, sometimes we denote the dimension vector $\alpha+\tau^{-1}\alpha$ also by $ac$. One importance of this weight was discovered by Schofield in \cite[Theorem 6.1]{S2}.

\begin{proposition} $\alpha$ is a Schur root if and only if it is $\sigma_{ac}$-stable.
\end{proposition}

\section{Bimodules and Stacks} \label{S:bimodule}
Suppose that an affine algebraic $k$-group $G$ acts on a $k$-scheme $X$.
Recall that the quotient stack $[X/G]$ is the category whose objects are principal $G$-bundles $\pi: P\to S$ together with a $G$-equivalent morphism $\phi:P\to X$. Morphisms are Cartesian diagram
such that $\phi f=\phi'$.
$$\xymatrix @C=2.5pc {
P \ar[rr]^{\phi} \ar[d]^{\pi} \ar[dr]^{f} && X\\
S \ar[dr]^{g} & P' \ar[d]^{\pi'} \ar[ur]_{\phi'} & \\
& S'
}$$

This is a {\em category fibred on groupoids} with structure functor $p$ sending the above morphism to $S'\xrightarrow{g}S$.
The fibre $p^{-1}(S)$ is a groupoid called $S$-points of $[X/G]$, denoted by $[X/G](S)$. Clearly the $k$-points are in one-to-one correspondence with the orbits of the action.
In particular, a scheme is naturally a quotient stack with $G=\{e\}$.

A morphism of quotient stacks $F:[X/G]\to [Y/H]$ is a functor between the categories commuting with the structure functors.
The stack $[X/G]$ is an {\em Artin stack} with the {\em atlas} the natural morphism $X\to [X/G]$ \cite[Example 2.29]{G}.

Let $A$ be some path algebra $kQ$. By the moduli stack $\Mod_\alpha(A)$, we mean the quotient stack $[\Rep_\alpha(A)/\GL_\alpha]$.
Let $B$ be another $k$-algebra, and $T$ be an $(A,B)$-bimodule. We assume that $T=\bigoplus T_i$ is a basic $A$-module with $\End_A(T)=B$. We fix for each $T_i$ a projective resolution $0\to P_1^i\to P_0^i\to T_i\to 0$. We also fix for each $\Hom_A(T_i,T_j)$ a basis $h_{ij}$ and its lift $\tilde{h}_{ij}$ through the following diagram:
$$\xymatrix @C=2.5pc {
0 \ar[r] & P_1^i \ar[r] \ar[d]  & P_0^i \ar[r]^{} \ar[d]^{\tilde{h}_{ij}} &  T_i \ar[d]^{h_{ij}} \ar[r] & 0\\
0 \ar[r] & P_1^j \ar[r]  & P_0^j \ar[r] & T_j \ar[r] & 0\\
}$$
\begin{definition} A representation $M$ is called {\em $T$-regular} if applying $\Hom_A(-,M)$ to a projective resolution $0\to P_1\to P_0\to T\to 0$, we get a surjective map $\Hom_A(P_0,M)\twoheadrightarrow\Hom_A(P_1,M)$. A dimension vector $\alpha$ is called $T$-regular if there is a $T$-regular representation in $\Rep_\alpha(A)$. We denote the set of all $T$-regular representations in $\Rep_\alpha(A)$ by $\Rep_\alpha(T^{\reg})$.
\end{definition}
Note that being $T$-regular is an open condition and does not depend on the choice of the projective resolution.
For simplicity, we assume that $\alpha$ is $T$-regular. This is equivalent to say that for each $i$, $\Hom(P_0^i,M)\to\Hom(P_1^i,M)$ is surjective for general $M\in\Rep_\alpha(A)$.
Later we will see that this is always the case for functors interested to us.

Let $M_{(n,m)}$ be the affine space of $n\times m$ matrices, and $M_{(n,m)}^0$ be the subvariety of full-rank matrices.

\begin{lemma} There is a morphism $M_{(n,n-1)}^0\to k^n$ sending a matrix $M$ to a non-zero vector $v$ such that $vM=0$.
\end{lemma}

\begin{proof} The $i$-th coordinates of $v$ can be chosen as the $i$-th maximal minor of $M$ with coefficient $(-1)^i$.
\end{proof}

\begin{corollary} \label{C:nullbase} For any $x\in M_{(n,n-k)}^0$ there is a neighborhood $U_x$ of $x$ with a morphism $U_x\to M_{(k,n)}^0$ sending a matrix $M$ to a basis of the null space of $M$.
\end{corollary}

\begin{proof} We prove by induction on $k$. The lemma is for $k=1$. Since $x$ has full rank, there is some non-vanishing maximal minor $m_I$ of $x$. Let $U_I$ be the open set defined by non-vanishing of the $I$-th maximal minor. We can add a fixed vector to any $y\in U_I$ such that the $n\times n-k+1$ matrix is still of full rank. Then our result follows from the induction.
\end{proof}

\begin{lemma} \label{L:open} Let $M$ be any $T$-regular representation with $\dim\Hom_A(T_i,M)=\beta_i$.
The functor: $\Hom_A(T,-): \module A\to \module B$ induces a morphism respecting orbits from an open neighborhood $U_M$ of $M$ to $\Rep_\beta(B)$.
\end{lemma}

\begin{proof} By Corollary \ref{C:nullbase}, we have an open neighborhood $U_M$ of $M$ such that there is a morphism from $U_M$ sending $N\in U_M$ to the null space of $\Hom_A(P_0,N)\twoheadrightarrow\Hom_A(P_1,N)$. The null space is nothing but $\Hom_A(T,N)$. The above diagram with fixed lift $\tilde{h}_{ij}$ allow us to define an element in $\Rep_\beta(B)$. This morphism respects orbits because it lifts the functor between module categories.
\end{proof}

\begin{lemma} \label{L:morphism} The functor: $\Hom_A(T,-): \module A\to \module B$ induces a morphism from $\Rep_\alpha(T^{\reg})$ to $\Mod_\beta(B)$.
\end{lemma}

\begin{proof} Let $\pi_M: U_M\to \Mod_\beta(B)$ be the morphism constructed in Lemma \ref{L:open} composed with the quotient map $\Rep_\beta(B)\to \Mod_\beta(B)$.
Since all such morphisms respect orbits, they agree on $\Rep_\alpha(T^{\reg})$.
\end{proof}

\section{Orthogonal Projection} \label{S:OP}

Let us review some basic facts about the orthogonal projections, first at the level of unimodular bilinear form and later at the level of module category. Let $V$ be a unimodular lattice with bilinear form $\innerprod{-,-}$ and $\epsilon$ be a root of $V$, i.e., $\innerprod{\epsilon,\epsilon}=1$. The left (resp. right) orthogonal lattice $^\perp\epsilon$ (resp. $\epsilon^\perp$) is the sublattice
$$\{\alpha\in V\mid\innerprod{\alpha,\epsilon}=0\} \text{ (resp. }\{\alpha\in V\mid\innerprod{\epsilon,\alpha}=0\}).$$
The Coxeter transformation $\tau$ is the unique linear transformation determined by $\innerprod{\alpha,\beta}=-\innerprod{\beta,\tau\alpha}$. Clearly $^\perp \epsilon=\tau^{-1}\epsilon^\perp$, so we can focus on the left one. Following \cite{L}, we define the orthogonal projection $\R_\epsilon:V\to\Perp\epsilon$ and the dual projection $\R_\epsilon^\vee:V\to\Perp\epsilon$ by
$$\R_\epsilon(\alpha)=\alpha-\innerprod{\alpha,\epsilon}\epsilon \quad\text{ and }\quad \R_\epsilon^\vee(\alpha)=\alpha+\innerprod{\alpha,\epsilon}\tau^{-1}\epsilon.$$
Clearly for $\alpha\in \Perp\epsilon$ and $\beta\in V$, we have that
$$\innerprod{\alpha,\beta}=\innerprod{\alpha,\R_\epsilon\beta} \text{ and } \innerprod{\beta,\alpha}=\innerprod{\R_\epsilon^\vee\beta,\alpha}.$$
A direct calculation can verify that $\innerprod{\alpha,\beta}=-\innerprod{\R_\epsilon^\vee\tau^{-1}\beta,\alpha}=-\innerprod{\beta,\R_\epsilon\tau\alpha}$ for $\alpha,\beta\in\epsilon^\perp$. In particular, $\R_\epsilon\tau$ is the Coxeter transformation for $\Perp \epsilon$ and its inverse is $\R_\epsilon^\vee\tau^{-1}$ \cite[Corollary 18.1]{L}.

Let $Q$ be a finite quiver without oriented cycles. For two $kQ$-module $M$ and $N$, $M$ is said to be left {\em orthogonal} to $N$ denoted by $M\perp N$ if $\Hom_Q(M,N)=\Ext_Q(M,N)=0$. In this case we also say that $N$ is right orthogonal to $M$. The left {\em orthogonal category} $\Perp N$ is the abelian subcategory $\{M\in\Mod(Q)\mid M\perp N\}$. A $kQ$-module $E$ is called {\em exceptional} if $\Hom_Q(E,E)=k$ and $\Ext_Q(E,E)=0$, so the dimension vector of $E$ corresponds to a {\em real} Schur root $\epsilon$.
We specialize some general results in \cite{GL} (see also \cite{S1}) to the quiver case.

\begin{lemma} \cite[Proposition 3.2, 3.5]{GL} \label{L:OP} $^\perp E$ is a reflective subcategory of $\Mod(Q)$, i.e.,
there is a functor $\R_E:\Mod(Q)\to\Perp E$ right adjoint to the inclusion functor
$^\perp E\to\Mod(Q)$. In particular, ${\R}_E$ is left exact and compatible with injective
presentations.
\end{lemma}

\begin{proof}[Sketch of proof.] It is useful to recall the construction in \cite{GL}, which is depicted by the following diagram. The row is the {\em universal extension} and the column is the {\em universal homomorphism}.
$$\xymatrix{
& \R_E(M) \ar@{_{(}->}[d] \ar[dr]^{r_M} & \\
eE \ar@{^{(}->}[r] & M' \ar[d] \ar@{->>}[r] & M\\
& hE &
}$$
Here, the universal extension means the extension universal with respect to the property that the connecting morphism $\Hom_Q(eE,E)\xrightarrow{\delta}\Ext_Q(M,E)$ in the long exact sequence is an isomorphism. Similarly, we mean by the universal homomorphism.

Alternatively, we can change the order of taking the universal extension and the universal homomorphism, but this leads to the same construction.
$$\xymatrix{
eE \ar@{^{(}->}[d] &  &  \\
\R_E(M) \ar@{->>}[d] \ar[dr]^{r_M} & &  \\
M'' \ar@{^{(}->}[r] & M \ar[r] & hE
}$$
Note that the composition $r_M$ is the universal homomorphism from an object of $^\perp E$ to $M$.
\end{proof}

\begin{definition} The {\em basic set} $b(M)$ of a module $M$ consists of all the non-isomorphic direct summands of $M$. We call $M$ {\em basic} if its direct summands are all pairwise non-isomorphic.

A $kQ$-module $T$ is called {\em partial tilting} if $\Ext_Q(T,T)=0$. A partial tilting module is called {\em tilting} if its direct summands generate the category $\Mod(Q)$. This is equivalent to say that the cardinality of the basic set of $T$ is equal to the number of vertices $|Q_0|$ \cite[Proposition VI.4.4]{ASS}.
\end{definition}

\begin{remark} \label{R:proj} If $I_v$ is the indecomposable injective module corresponding to a vertex $v$, then
applying $\R_E$ to $I_v$ is particularly simple. If $E$ is not injective, then $\Hom_Q(I_v,E)=0$; otherwise $\Ext_Q(I_v,E)=0$.
Clearly, $\R_E(kQ^*)=\oplus_{v\in Q_0}\R_E(I_v)$ is partial tilting. Applying $\Hom_Q(E,-)$ to the construction of Lemma \ref{L:OP}, we see that $\Ext_Q(E,\R_E(kQ^*))=0$, so $\R_E(kQ^*)\oplus E$ is a tilting module for $\Mod(Q)$. Let $I_b$ be a direct sum of elements in $b(\R_E(kQ^*))$, then the quiver of $\End_Q(I_b)$ has $|Q_0|-1$ vertices and no oriented cycles.
\end{remark}

\begin{corollary} \cite[Theorem 2.3]{S1} \label{C:OC}
The orthogonal category $\Perp E$ is (Morita) equivalent to representations of a quiver $Q_E$ having no oriented cycles with $|Q_0|-1$ vertices. The inverse functor $\iota_E: \Mod(Q_E)\to\Perp E$ is a full exact embedding into $\Mod(Q)$.
\end{corollary}

We called the functor $\R_E$ in Lemma \ref{L:OP} the (right) {\em orthogonal projection through} $E$ and its composition with the equivalence in Lemma \ref{C:OC} the (right) orthogonal projection {\em to} $Q_E$, denoted by $\R_{Q_E}:\Mod(Q)\to\Mod(Q_E)$. Sometimes we do not distinguish $\Perp E$ and $\Mod(Q_E)$ if no confusion is possible.

If $E$ is not injective, then everything above has a dual statement for the right orthogonal category $E^\perp$ and left orthogonal projection $\L_E$ and $\L_{Q_E}$.
Let $\tau$ be the classical AR-transformation \cite[IV.2]{ASS} on $\Mod(Q)$. By the AR-duality \cite[Theorem IV.2.13]{ASS}, $^\perp E=\tau^{-1} E^\perp$. We define the dual right orthogonal projection $\R_E^\vee:=\L_{\tau^{-1} E}$, the left orthogonal projection through $\tau^{-1} E$.
Then the dual of Lemma \ref{L:OP} implies that $b(\R_E^\vee(kQ))$ contains all the indecomposable projective module in $^\perp E$. We denote by $P_b$ the direct sum of all elements in $b(\R_E^\vee(kQ))$.

It is well-known \cite{W} that an adjoint pair between module categories must be representable by a bimodule.
The next lemma says that the bimodule for $\R_E$ is explicitly given by $b(\R_E^\vee(kQ))$.

\begin{lemma} \label{L:TD}  $\R_E(M)=\Hom_Q(\R_E^\vee(kQ),-)$, so $\R_{Q_E}(M)=\Hom_Q(P_b,-)$.
\end{lemma}

\begin{proof} First, we treat the case when $M=I$ is injective. Apply the functor $\Hom_Q(\R_E^\vee(kQ),-)$ to the exact sequence $0\to eE\to \R_E(I)\to I\to 0$, then we get
$$\Hom_Q(\R_E^\vee(kQ),eE)\rightarrow\Hom_Q(\R_E^\vee(kQ),\R_E(I))\rightarrow\Hom_Q(\R_E^\vee(kQ),I)\rightarrow\Ext_Q(\R_E^\vee(kQ),eE).$$
Since the first and last term vanish and $\R_E^\vee$ is left adjoint to the inclusion functor, $$\Hom_Q(\R_E^\vee(kQ),\R_E(I))=\Hom_Q(kQ,\R_E(I))=\R_E(I).$$ Hence $\Hom_Q(\R_E^\vee(kQ),I)\cong \R_E(I)$.

Now for any representation $M$, we apply $\Hom_Q(\R_E^\vee(kQ),-)$ to its injective resolution $0\to M\to I_0\to I_1\to 0$. Comparing
$$0\to \Hom_Q(\R_E^\vee(kQ),M)\rightarrow\Hom_Q(\R_E^\vee(kQ),I_0)\rightarrow\Hom_Q(\R_E^\vee(kQ),I_1)$$
with $0\to \R_E(M)\to\R_E(I_0)\to\R_E(I_1)$,
We find that $\Hom_Q(\R_E^\vee(kQ),M)\cong\R_E(M)$ by the natural functoriality.
\end{proof}

The equivalence in Corollary \ref{C:OC} induces a linear isometry between the $K_0$-groups equipped with the Euler forms. In particular for $M,N\in\Perp E$,
$$\dim_Q M=\dim_Q N\iff\dim_{Q_E}\R_{Q_E}(M)=\dim_{Q_E}\R_{Q_E}(N).$$
In general, we consider any dimension vector $\alpha$. Under the functor $\R_{Q_E}$, general elements in $\Rep_\alpha(Q)$ maps to $\Mod_{\alpha_E}(Q_E)$ for some dimension vector $\alpha_E$. We also have the dual notion $\alpha_E^\vee$ for the functor $\R_{Q_E}^\vee$.

Let $\epsilon$ be a real root of $Q$. We denote by $\R_\epsilon$ the right orthogonal projection on the $K_0$-group of $\Mod(Q)$ equipped with the Euler form $\innerprod{-,-}_Q$, i.e., $\R_\epsilon(\alpha)=\alpha-\innerprod{\alpha,\epsilon}_Q\epsilon$. Similarly, we write $\R_{Q_\epsilon}$ for the composition of $\R_\epsilon$ with the linear isometry induced by the equivalence in Lemma \ref{C:OC} and $\iota_\epsilon$ for the inverse of this linear isometry.
To simplify our notation, sometimes we write $\tilde{\alpha}_\epsilon:=\R_\epsilon(\alpha),\alpha_{\epsilon}:=\R_{Q_\epsilon}(\alpha),$ and $\alpha^\epsilon:=\iota_\epsilon(\alpha)$. We need the obvious dual notions $\R_\epsilon^\vee,\R_{Q_\epsilon}^\vee,$ and $\tilde{\alpha}_\epsilon^\vee, \alpha_\epsilon^\vee$ later.

\begin{definition} A representation $M\in\Rep_\alpha(Q)$ is called {\em $E$-regular} if
$$\dim_Q\R_E(M)=\tilde{\alpha}_\epsilon \text{ or equivalently }\dim_{Q_E}\R_{Q_E}(M)=\alpha_{\epsilon}.$$
We denote the set of all $E$-regular representations by $E^{\reg}$ and $E^{\reg}\cap\Rep_\alpha(Q)$ by $\Rep_\alpha(E^{\reg})$.
\end{definition}

\begin{lemma} \label{L:Ereg} The following are equivalent for a representation $M$ \begin{enumerate}
\item $M$ is $E$-regular.
\item The universal homomorphism $M\to hE$ is surjective.
\item The functor $\R_E$ is exact on any injective resolution $0\to M\to I_0\to I_1\to 0$.
\item $\Ext_Q(\R_E^\vee(kQ),M)=0$.
\end{enumerate}
\end{lemma}

\begin{proof} $(1)\Leftrightarrow (2)$ and $(3)\Leftrightarrow(4)$ follows from the construction of Lemma \ref{L:OP} and Lemma \ref{L:TD} respectively.
For $(1)\Rightarrow(3)$, we need to show that $\R_E(I_0)\to \R_E(I_1)$ is surjective. If not, let $C$ be the cokernel, which must be injective in $^\perp E$. Since $M,I_0,I_1$ are all $E$-regular, we see that the dimension of $C$ is a multiple of $\epsilon$, which is impossible. For $(3)\Rightarrow(2)$, we consider the following diagram
$$\xymatrix @C=2.5pc {
& K \ar@{^{(}->}[d] \ar[r] & \R_E(M) \ar@{^{(}->}[d] \ar[r]^{r_M} & M  \ar@{^{(}->}[d] \\
0 \ar[r] & e_0E^{} \ar[r]^{} \ar[d]^{f} & \R_E(I_0) \ar[r]^{r_{I_0}} \ar[d]^{} & I_0 \ar[d]^{} \ar[r] & 0\\
0 \ar[r] & e_1E^{} \ar[r] \ar@{->>}[d] &  \R_E(I_1) \ar[r]^{r_{I_1}} \ar[d] &  I_1 \ar[r] \ar[d] & 0 \\
& C & 0 & 0
}$$
By the snake lemma, we get an exact sequence $0\to K \to \R_E(M)\xrightarrow{r_M} M\to C\to 0$ appeared in the construction of Lemma \ref{L:OP}. Since $E$ is exceptional, the cokernel $C$ of $f$ must be a direct sum of $E$'s. Hence $M\to C$ is the universal homomorphism map and it is surjective.
\end{proof}

According to Lemma \ref{L:Ereg}.(4), $E$-regular is the same as $\R_E^\vee(kQ)$-regular of Section \ref{S:bimodule}. The next corollary follows easily from the adjoint property of $\R_E$ and the above diagram.

\begin{corollary}
If $M$ is $E$-regular, then $\Ext_{Q}(-,\R_E(M))=\Ext_Q(-,M)$ on $^\perp E$.
\end{corollary}

In general, we have $\Ext_{Q}(-,\R_E(M))\subseteq \Ext_Q(-,M)$. Here we prove the dual statement.
\begin{lemma} \label{L:proj_ss} For any $M\in\Mod(Q)$, $\Ext_{Q}(\R_{E}^\vee(M),-)\subseteq\Ext_Q(M,-)$ on $\Perp E$. So $M\perp N$ implies $\R_{E}^\vee(M)\perp N$ for $N\in \Perp E$.
\end{lemma}

\begin{proof}
Take a projective resolution of $M:0\to P_1\to P_0\to M\to 0$, then $\R_E^\vee(P_1)\to \R_E^\vee(P_0)\to \R_E^\vee(M)\to 0$ is exact.  Let us look at the following diagram.
$$\xymatrix @C=0.5pc {
0 \ar[r] & \Hom_Q(M,N) \ar[r]\ar@{=}[d] &  \Hom_Q(P_0,N) \ar[r]\ar@{=}[d] & \Hom_Q(P_1,N) \ar[r] \ar@{=}[d] & \Ext_Q(M,N) \ar[r] \ar@{:}[d] & 0 \\
0 \ar[r] & \Hom_{Q}(\R_E^\vee(M),N) \ar[r] & \Hom_{Q}(\R_E^\vee(P_0),N) \ar[r] & \Hom_{Q}(\R_E^\vee(P_1),N) \ar[r] & \E \ar[r] & 0 \\
}$$
Since $N\in \Perp E$, the adjunction property gives the first three vertical isomorphisms, which implies the last isomorphism. Clearly $\Ext_{Q}(\R_E^\vee(M),N)$ is a subspace of $\E$.
\end{proof}

%
%

Let us specialize the above discussion to the orthogonal case. Since any representation in $\Perp E$ is trivially $E$-regular, we get a surjective morphism $\Rep_\alpha(E^{\reg})\to\Mod_{\alpha_\epsilon}(Q_E)$ from Lemma \ref{L:morphism}. By a slight abuse of notation, we denote this morphism still by $\R_{Q_E}$.

\begin{definition} \label{D:exc}
We say that $E$ is right orthogonal to a dimension vector $\alpha$, denoted by $\alpha\perp E$, if $\Rep_\alpha(\Perp E):=\Rep_\alpha(Q)\cap {^\perp} E$ is non-empty. We say $E$ is right exceptional to $\alpha$ if in addition $\mc{E}_\alpha:=\Rep_\alpha(E^{\reg})\setminus\Perp E$ is non-empty. We call $\mc{E}_\alpha$ an {\em exceptional set} in $\Rep_\alpha(Q)$ with respect to $E$. It has codimension one in $\Rep_\alpha(Q)$ because it is open in $C_E$.
\end{definition}

Although we do not know if there are only finitely many $E$ orthogonal to a fixed $\alpha$, Lemma \ref{L:Ereg}.(2) does guarantee that there are only finitely many $E$ exceptional to $\alpha$.
The following lemma is a nice application of Luna's \'{e}tale slice \cite{Lu}.

\begin{lemma} \label{L:homofibre}\cite[Theorem 3.2]{S1} If $\alpha\perp E$, then $\Rep_\alpha(\Perp E)$ is isomorphic to the homogeneous fibre space  $\GL_\alpha\times_{\GL_{\alpha_\epsilon}}\Rep_{\alpha_\epsilon}(Q_E)$.
\end{lemma}

\begin{example} \label{Ex:1}
Consider the quiver $B^1$
$$\extendedtameAtwo{2}{1}{3}{b_1}{a_1}{c_1}{d_1}$$ with dimension vector $\alpha=(1,1,1)$. Then the only exceptional representation $E$ right orthogonal to $\alpha$ is the simple $S_1$.
Let us compute $B^1_E$ and $\alpha_\epsilon$. We observe that $P_1,P_3\in\Perp E$ and $\R_E^\vee(P_2)=P_1$. So $\Rep_\alpha(E^{\reg})$ is the whole $\Rep_\alpha(B^1)$. Moreover $B^1_E$ is the three-arrow Kronecker Quiver $\Theta_3$ with $\alpha_\epsilon=(1,1)$. Let us denote the three arrows by $a,b,c$, then the morphism $\Rep_\alpha(B^1)\to\Rep_{\alpha_\epsilon}(\Theta_3)$ can be explicitly described as $A=D_1A_1,B=C_1A_1$, and $C=B_1$.
Note that $\R_E$ crashes $\mc{E}_\alpha$ defined by $A_1=0$ to a single module $\tilde{N}_1$ represented by $\{A_1=B_1=1,C_1=D_1=0\}$.
So $\R_{Q_E}(\mc{E}_\alpha)$ is represented by $\{A=B=0,C=1\}$.

Consider the quiver $B^2$
$$\extendedtameAthree{1}{4}{3}{2}{a_2}{b_2}{e_2}{c_2}{d_2}$$ with dimension vector $\alpha=(1,1,1,1)$. Then there are three exceptional representations right orthogonal to $\alpha$. They are \begin{align*}
& 0\to E_v\to I_3\to I_v\to 0,\quad v=1,2\\
& 0\to E_3\to I_4\to I_3\to 0.
\end{align*}
Let $E=E_2$. We observe that $P_2,P_4\in \Perp E$, $\R_{E}^\vee(P_3)=P_2$, and $\R_{E}^\vee(P_1)=P_3\to P_2\oplus P_1$.
So $B^2_{E}=B^1$ with $\alpha_\epsilon=(1,1,1)$. The $E$-regularity is defined by the condition that $A_2\neq 0$ or $C_2\neq 0$. The morphism $\Rep_\alpha(E^\reg)\to\Mod_{\alpha_\epsilon}(B^1)$ can be explicitly described as $$\begin{cases}
A_1=-A_2,\\ B_1=C_2B_2,\\ C_1=E_2C_2,\\ D_1=D_2. \end{cases}$$

By symmetry $B^2_{E_1}=B^1$ and $\alpha_{\epsilon_1}=(1,1,1)$. However, $B^2_{E_3}$ is the quiver $\Theta_{2,2}:$
$$\productthreevertexcomplex{1}{2}{3}{a}{b}{c}{d}.$$ with $\alpha_{\epsilon_3}=(1,1,1)$. We claim that there is no exceptional representation orthogonal to $(1,1,1)$ for $\Theta_{2,2}$. If there is one with dimension $\epsilon=(x,y,z)$, then $\innerprod{\alpha_{\epsilon_3},\epsilon}_Q=0$ and $\innerprod{\epsilon,\epsilon}_Q=1$. Elementary calculation can show that $2z(z-y)=1$, which is impossible. So the answer to the question in \cite[introduction]{S3} is negative in general.

Consider the quiver $B^3$
$$\extendedtameAfour{b_3}{a_3}{d_3}{c_3}{e_3}{f_3}$$ with dimension vector $\alpha=(1,1,1,1,1)$.
Then there are six exceptional representations right orthogonal to $\alpha$. They are $$0\to E_{uv} \to I_v\to I_u\to 0,$$ where $u=1,2,3,v=4,5$.
Let $E=E_{35}$. We observe that $P_3,P_4\in \Perp E$, $\R_{E}^\vee(P_5)=P_3$, and $\R_{E}^\vee(P_v)=P_5\to P_3\oplus P_v$ for $v=1,2$. So $B^3_{E}=B^2$ with $\alpha_\epsilon=(1,1,1,1)$. The morphism $\Rep_\alpha(E^{\reg})\to\Mod_{\alpha_\epsilon}(B^2)$ can be explicitly described as $$\begin{cases}
A_2=-A_3,\\ B_2=B_3F_3,\\ C_2=-C_3,\\ D_2=D_3F_3,\\ E_2=E_3.\end{cases}$$

Consider the quiver $B^4$
$$\extendedtameDfour{a_4}{b_4}{c_4}{d_4}{e_4}$$ with dimension vector $\alpha=(1,1,1,1,1,2)$.
Then there are ten exceptional representations right orthogonal to $\alpha$. They are $$0\to E_{uv}\to I_6\to I_u\oplus I_v\to 0,$$ where $\{u,v\}\subset\{1,2,3,4,5\}$.
Let $E=E_{45}$. We observe that $P_4,P_5\in \Perp E$, $\R_{E}^\vee(P_6)=P_4\oplus P_5$, and $\R_{E}^\vee(P_v)=P_6\to P_5\oplus P_4\oplus P_v$ for $v=1,2,3$. So $B^4_{E}=B^3$ with $\alpha_\epsilon=(1,1,1,1,1)$. The morphism $\Rep_\alpha(E^\reg)\to\Mod_{\alpha_\epsilon}(B^3)$ can be explicitly described as $$\begin{cases}
A_3=-\det(A_4,D_4),\\
B_3=\det(A_4,E_4),\\
C_3=-\det(B_4,D_4),\\
D_3=\det(B_4,E_4),\\
E_3=\det(C_4,E_4),\\
F_3=-\det(C_4,D_4).\end{cases}$$
\end{example}

\begin{example} \label{Ex:2}
Consider the quiver $B^{4,1}$
$$\extendedtameDfive{a_5}{b_5}{c_5}{d_5}{e_5}{f_5}$$ with dimension vector $\alpha=(1,1,1,1,1,2,2)$. Then there are eighteen representations exceptional to $\alpha$. The ones interesting to us are \begin{align*}
& 0\to E_1\to I_7\to I_6\to 0\\
& 0\to E_2\to I_6\to I_5\oplus I_4\to 0
\end{align*}

To compute $B^{4,1}_{E_1}$, we observe that all $P_i$'s except $P_6$ are in $\Perp E_1$.
So $B^{4,1}_{E_1}=B^4$ with $\alpha_{\epsilon_1}=(1,1,1,1,1,2)$. The morphism $\Rep_\alpha(\Perp E_1)\to\Rep_{\alpha_{\epsilon_1}}(B^4)$ can be explicitly described as $$\begin{cases}
A_4=A_5,\\
B_4=B_5,\\
C_4=C_5,\\
D_4=D_5F_5,\\
E_4=E_5F_5.\end{cases}$$
We leave it to interested reader to verify that $B^{4,1}_{E_2}$ is also equal to $B^4$ and those are the only two among the eighteen representations.
\end{example}

\begin{example} \label{Ex:3} Consider quiver $C_3^6:$ $$\threevertextwoone{1}{2}{3}{a}{b}{c}$$ with dimension vector $\alpha=(3,4,1)$, then the only representation right exceptional to $\alpha$ is $E=(2,3,0)$. By computation, we have that $\R_E^\vee(P_1)=3(2P_2\to 3P_1),\R_E^\vee(P_2)=4(2P_2\to 3P_1)$, and $\R_E^\vee(P_3)=P_3$. So $(C_3^6)_E$ is the four-arrow Kronecker quiver $\Theta_4$ with $\alpha_\epsilon=(1,1)$. Let us denote the four arrows by $a_E,b_E,c_E,d_E$. As a good exercise, one can check that $2P_2\to 3P_1$ can be represented by the matrix $\sm{a & 0 & b \\ 0 & b & a}$, and that the morphism $\Rep_\alpha(E^{\reg})\to\Mod_{\alpha_\epsilon}(\Theta_4)$ can be represented by $A_E=\det\sm{A & 0 & AC \\ 0 & B & 0 \\ B & A & 0}, B_E=\det\sm{A & 0 & BC \\ 0 & B & 0 \\ B & A & 0}, C_E=\det\sm{A & 0 & 0 \\ 0 & B & AC \\ B & A & 0}, D_E=\det\sm{A & 0 & 0 \\ 0 & B & BC \\ B & A & 0}$.
The representations in $\mc{E}_\alpha$ have the following normal form: $A=\sm{1 & 0 & 0 & 0\\ 0 & 1 & 0 & 0\\ 0 & 0 & 0 & s}, B=\sm{0 & 1 & 0 & *\\ 0& 0 & 1 & 0 \\ 0 & 0 & 0 & t}$ and $C=(*,*,*,-1)^{\T}$. Plug in and we get the image of $\mc{E}_\alpha$ has the form \begin{equation} \label{eq:cubic} A_E=st^2, B_E=t^3, C_E=s^3, D_E=s^2t. \end{equation}
This example is taken from \cite{F2}. We will continue these examples later.
\end{example}

\section{The Exceptional Set $\mc{E}_\alpha$} \label{S:FR}
\begin{definition}
For any dimension vector $\gamma$, we define $\Hom(k^\alpha,k^\beta)_\gamma$
$$=\{(M_1,N_1,\phi)\in \Gr \binom{\alpha}{\alpha-\gamma}\times \Gr \binom{\beta}{\gamma} \times \Hom(k^\alpha,k^\beta)\mid \Ker\phi=M_1, \Img\phi=N_1\},$$
$\Hom_Q(\alpha,\beta)_\gamma$
$$=\{(M,N,M_1,N_1,\phi)\in \Rep_\alpha(Q)\times\Rep_\beta(Q)\times\Hom(k^\alpha,k^\beta)_\gamma \mid \phi\in\Hom_Q(M,N)\}.$$
\end{definition}

\begin{lemma}
$q:\Hom(k^\alpha,k^\beta)_\gamma \to \Gr\binom{\alpha}{\alpha-\gamma}\times \Gr\binom{\beta}{\gamma}$ is a principal $\GL_\gamma$-bundle. In particular, $\Hom(k^\alpha,k^\beta)_\gamma$ is smooth and irreducible.\\
$r:\Hom_Q(\alpha,\beta)_\gamma\to \Hom(k^\alpha,k^\beta)_\gamma$ is a vector bundle with fibre
$$\bigoplus_{a\in Q_1}(\Hom(k^{\gamma(ta)},k^{\gamma(ha)})\oplus \Hom(k^{(\beta-\gamma)(ta)},k^{\beta(ha)})\oplus \Hom(k^{\alpha(ta)},k^{(\alpha-\gamma)(ha)})).$$
In particular, $\Hom_Q(\alpha,\beta)_\gamma$ is smooth and irreducible with \begin{equation}
\dim(\Hom_Q(\alpha,\beta)_\gamma)=\dim\Rep_\alpha(Q)+\dim\Rep_\beta(Q)+\innerprod{\alpha,\beta}-\innerprod{\alpha-\gamma,\beta-\gamma}.
\end{equation}
\end{lemma}

\begin{proof}
(Sketch) For any $\iota\in\Hom(k^\gamma,k^\alpha)$ injective and $\pi\in\Hom(k^\beta,k^\gamma)$ surjective, define $V_{\iota\pi}=\{(M_1,N_1,\phi)\in\Hom(k^\alpha,k^\beta)_\gamma\mid \pi\phi\iota$ is isomorphism\} and $U_{\iota\pi}=\{(M_1,N_1)\in\Gr \binom{\alpha}{\alpha-\gamma}\times \Gr \binom{\beta}{\gamma}\mid k^\alpha=M_1+\Img(\iota), k^\beta=N_1+\Ker(\pi)\}$. Then \{$U_{\iota\pi}$\} is an open cover of $\Gr \binom{\alpha}{\alpha-\gamma}\times \Gr \binom{\beta}{\gamma}$ and  \{$V_{\iota\pi}$\} is an open cover of $\Hom(k^\alpha,k^\beta)_\gamma$ such that $V_{\iota\pi}=U_{\iota\pi}\times\GL_\gamma$. Moreover, one can show that $r$ is trivial over $V_{\iota\pi}$.
\end{proof}

Now we assume that $\innerprod{\alpha,\beta}_Q=0$. Let $C$ be the subscheme of $\Rep_\alpha(Q)\times\Rep_\beta(Q)$ defined by $c(M,N):=\det\phi_N^M$ \eqref{eq:canseq}, so its reduced part is $$C_{\red}=\{(M,N)\in\Rep_\alpha(Q)\times\Rep_\beta(Q)\mid\hom_Q(M,N)>0\}.$$
\begin{definition}
We call the (non-trivial) general rank $\gamma$ on any irreducible component of $C$ a {\em fundamental rank for} $(\alpha,\beta)$.
\end{definition}

\begin{lemma} \label{L:frank} All the irreducible components of $C$ are parameterized by the fundamental ranks. We denote by $C_\gamma$ the component with fundamental rank $\gamma$, then
\begin{equation} \label{eq:dimhomC} \ext_Q(C_\gamma)=\hom_Q(C_\gamma)=1-\innerprod{\alpha-\gamma,\beta-\gamma}. \end{equation}
\end{lemma}

\begin{proof}
Let $C_0$ be irreducible component of $C_{\red}$ and $\gamma$ be the general rank on $C_0$. Consider the projection $p:\Hom_Q(\alpha,\beta)_\gamma\to \Rep_\alpha(Q)\times\Rep_\beta(Q)$. $\overline{p(\Hom_Q(\alpha,\beta)_\gamma)}$ is irreducible and must equal to $C_0$.

If $(M,N)\in C_\gamma$ is in general position, then a general element in $\Hom_Q(M,N)$ has rank $\gamma$. But the fibre $p^{-1}(M,N)$ is $\{\phi\in\Hom_Q(M,N)\mid \rank\phi=\gamma\}$, so
\begin{align*}
\ext_Q(C_\gamma)=\hom_Q(C_\gamma)& =\dim p^{-1}(M,N)\\
& =\dim(\Hom_Q(\alpha,\beta)_\gamma)-\dim(\Rep_\alpha(Q)\times\Rep_\beta(Q))\notag\\
& =1-\innerprod{\alpha-\gamma,\beta-\gamma}.
\end{align*}
\end{proof}

We denote the following three sets respectively by $\Rep_{\gamma\hookrightarrow\alpha}(Q),\Rep_{\alpha\twoheadrightarrow\gamma}(Q)$, and $\Rep_{\gamma\oplus\alpha-\gamma}(Q)$:
\begin{align*}
&\{M\in\Rep_\alpha(Q)\mid M\text{ has a $\gamma$-dimensional subrepresentation}\},\\
&\{M\in\Rep_\alpha(Q)\mid M\text{ has a $\gamma$-dimensional quotient representation}\},\\
&\{M\in\Rep_\alpha(Q)\mid M\text{ has a direct sum decomposition of dimension $\gamma$ and $\alpha-\gamma$}\},
\end{align*}
The following lemma is implied in the proof of \cite[Theorem 2.1, 3.3]{S2}.

\begin{lemma} \label{L:codimgen} $\Rep_{\alpha\twoheadrightarrow\gamma}(Q)$ is a closed irreducible subvariety of codimension equal to $\ext_Q(\alpha-\gamma,\gamma)$ in $\Rep_\alpha(Q)$; $\Rep_{\gamma\oplus\alpha-\gamma}(Q)$ is irreducible of codimension equal to $\ext_Q(\gamma,\alpha-\gamma)+\ext_Q(\alpha-\gamma,\gamma)$ in $\Rep_\alpha(Q)$.
\end{lemma}

\begin{corollary} \label{C:exc} If $\alpha\perp E$ ,then the following are equivalent \begin{enumerate}
\item[(1)] $E$ is right exceptional to $\alpha$;
\item[(2)] $\epsilon$ is a fundamental rank for $(\alpha,\epsilon)$;
\item[(3)] $\ext_Q(\alpha-\epsilon,\epsilon)=1$, or equivalently $\hom_Q(\alpha-\epsilon,\epsilon)=0$;
\end{enumerate}
In this case, the exceptional set $\mc{E}_\alpha$ is irreducible and $\hom_Q(\mc{E}_\alpha,E)=1$. Moreover, $\Rep_{\alpha\twoheadrightarrow\epsilon}(Q)=\overline{\mc{E}_\alpha}$ and $\Rep_{\epsilon\hookrightarrow\alpha}(Q)=\overline{\GL_\alpha\cdot\R_E(\mc{E}_\alpha)}$.
\end{corollary}

\begin{proof} If $\mc{E}_\alpha$ is non-empty, then $\epsilon$ is clearly a general rank on $\mc{E}_\alpha\times \GL_\epsilon E$. But $\GL_\epsilon E$ is dense in $\Rep_\epsilon(Q)$, so $\mc{E}_\alpha\times \GL_\epsilon E$ has codimension one in $\Rep_\alpha(Q)\times \Rep_\epsilon(Q)$, and thus must be irreducible. Hence, $\epsilon$ is a fundamental rank for $(\alpha,\epsilon)$ and $\mc{E}_\alpha$ is irreducible. Now each representation on $\mc{E}_\alpha$ has a quotient representation of dimension $\epsilon$, so $\ext_Q(\alpha-\epsilon,\epsilon)=1$.
Conversely, if $\ext_Q(\alpha-\epsilon,\epsilon)=1$, then $\hom_Q(\alpha-\epsilon,\epsilon)=0$. Let $M$ be a general representation in $\Rep_{\alpha\twoheadrightarrow\epsilon}(Q)$. We can assume that the kernel $S$ of $M\twoheadrightarrow E$ is general in $\Rep_{\alpha-\epsilon}(Q)$.
We claim that $\Hom_Q(M,E)$ contains solely one epimorphism, otherwise $\hom_Q(S,E)=\hom_Q(M,E)-1>0$ contradicting $\hom_Q(\alpha-\epsilon,\epsilon)=0$. So this epimorphism is indeed the universal homomorphism.

The equality $\hom_Q(\mc{E}_\alpha,E)=1$ follows from the formula \eqref{eq:dimhomC}. $\Rep_{\alpha\twoheadrightarrow\epsilon}(Q)=\overline{\mc{E}_\alpha}$ is evident from Lemma \ref{L:codimgen}. So we can require the kernel of the universal homomorphism is in general position, and therefore $\GL_\alpha\cdot\R_E(\mc{E}_\alpha)$ is dense in $\Rep_{\epsilon\hookrightarrow\alpha}(Q)$.
\end{proof}


\begin{example} \label{Ex:0} It is possible that $E$ is right orthogonal to $\alpha$ with $\dim_Q E<\alpha$, but $E$ is not right exceptional to $\alpha$. Consider a reflected quiver $C_3^6:$ $$\Cthreesix{}{}{}$$ with $\alpha=(1,2,2)$, then the only representation right orthogonal to $\alpha$ is $E=(0,2,1)$. However, $\alpha-\epsilon=(1,0,1)$ and $\Hom_Q(S_3,E)=1$, so $\hom_Q(\alpha-\epsilon,\epsilon)=1$. Moreover, one can check that the only fundamental rank for $(\alpha,\epsilon)$ is $(0,1,1)\neq\epsilon$.
\end{example}

The proposition will be used in Section \ref{S:SC}.
\begin{proposition} \label{P:dimblowup} Suppose that for any dimension vector $\gamma<\alpha$, $\Rep_{\epsilon\hookrightarrow\alpha}(Q)\nsubseteq \overline{\Rep_{\gamma\oplus\alpha-\gamma}(Q)}$, then
$\hom_Q(\epsilon,\alpha-\epsilon)=0$, so $\ext_Q(\epsilon,\alpha-\epsilon)=-\innerprod{\alpha-\epsilon,\epsilon}_Q$.
\end{proposition}

\begin{proof} The proof is just a slight variation of that of \cite[Theorem 4.1]{S1}.
Suppose that $\hom_Q(\epsilon,\alpha-\epsilon)>0$, then let $\gamma$ be the general rank on $\Rep_{\epsilon}(Q)\times\Rep_{\alpha-\epsilon}(Q)$. We claim that either $\gamma=\epsilon$ or $\gamma=\alpha-\epsilon$. Suppose this is not the case. We take a general representations $N\in\Rep_{\alpha-\epsilon}(Q)$ with a general morphism $f$ from $E$ to $N$. Let $K,I,C$ be the kernel, image, and cokernel of $f$. Note that $K$ and $C$ can not both be zeros, otherwise $\alpha=2\epsilon$. Since $\ext_Q(\epsilon,\epsilon)=0$, this will contradicts our assumption by Lemma \ref{L:codimgen}. We assume that $K$ is nonzero, otherwise the proof goes similarly. We get an induced surjection $\Ext_Q(N,K)\twoheadrightarrow\Ext_Q(I,K)$, which corresponds to the following exact diagram of extensions:
$$\xymatrix @C=2.5pc {
0 \ar[r] & K \ar[r]^{} \ar@{=}[d] &  E \ar[r]^{\pi'} \ar[d]^{\iota'} & I \ar[r]  \ar@{_{(}->}[d]^{\iota} & 0 \\
0 \ar[r] & K \ar[r]^{} & M \ar[r]^{\pi} & N \ar[r]  & 0, \\
}$$ from which we can construct an exact sequence of representations
$$\xymatrix @C=2.5pc {
0 \ar[r] & E \ar[r]^{(\pi'\ -\iota')} &  I\oplus M \ar[r]^{\sm{\iota\\ \pi}} & N \ar[r] & 0.}$$
This exact sequence cannot split, otherwise $I\oplus M=E\oplus N$ contradicts the uniqueness of decomposition. So the sequence corresponds to a general extension in $\Ext_Q(N,E)$. This is because $\ext_Q(\alpha-\epsilon,\epsilon)=1$ and $N,E$ are general. Then a general representation in $\Rep_{\alpha\hookrightarrow\epsilon}(Q)$ has a decomposition of dimension $\gamma$ and $\alpha-\gamma$, which contradicts our assumption.
\end{proof}

\section{A Functorial Construction} \label{S:FC}

Let $\alpha$ be a non-isotropic imaginary Schur root and $\sigma_\beta=\innerprod{\beta,-}_Q$ be a weight. We assume that $\Rep_\alpha(Q)$ contains a $\sigma_\beta$-stable representation.
Let $E$ be an exceptional representation with $\dim_Q E=\epsilon$ and we assume it is right orthogonal to $\alpha$. Note that $\Rep_\alpha(\Perp E)=\Rep_\alpha\ss{\tau^{-1}\epsilon}{ss}(Q)$ because $\GL_\epsilon E$ is dense in $\Rep_\epsilon(Q)$.
We denote by $\Rep_\alpha\ss{\beta}{ss}(Q)\cap\Perp E$ (resp. $\Rep_\alpha\ss{\beta}{st}(Q)\cap\Perp E$) by $\Rep_\alpha\ss{\beta}{ss}(\Perp E)$ (resp. $\Rep_\alpha\ss{\beta}{st}(\Perp E$)).
It follows from Lemma \ref{L:proj_ss} that $\alpha_\epsilon$ is $\sigma_{\beta_E^\vee}$-semi-stable.

\begin{definition} A representation $M\in\Rep_\alpha\ss{\beta}{ss}(Q)$ is called {\em $E$-effective} if $\R_{Q_E}(M)$ is $\sigma_{\beta_E^\vee}$-semi-stable. Clearly, being $E$-effective is an open condition.
We denote the set of all $E$-effective $\sigma_\beta$-semi-stable $\alpha$-dimensional representations by $\Rep_\alpha\ss{\beta_E^\vee}{ss}(Q)$ and all $E$-effective $\sigma_\beta$-unstable representations in $\Rep_\alpha(\Perp E)$ by $\Rep_\alpha\ss{\beta_E^\vee}{us}(\Perp E)$.
\end{definition}

\begin{lemma} \label{L:pi_E_ss}
If $M\in \Rep_\alpha(\Perp E)$ is $\sigma_\beta$-$($semi-$)$stable, then $\R_{Q_E}(M)$ is $\sigma_{\beta_E^\vee}$-$($semi-$)$stable. In particular, $\alpha_\epsilon$ is $\sigma_{\beta_E^\vee}$-stable.
\end{lemma}

\begin{proof}
Let $M_1\in\Mod(Q_E)$ be a quotient representation of $\R_{Q_E}(M)$, then we have that $\sigma_{\beta_E^\vee}(\dim_{Q_E}M_1)\geqslant\sigma_\beta(\dim_Q\iota_E(M_1))$ by Lemma \ref{L:proj_ss}. Since $\iota_E$ is an exact embedding, $\iota_E(M_1)$ is also a quotient representation of  $\iota_E\R_{Q_E}(M)$. But $\iota_E\R_{Q_E}(M)=M$ because $M\in\Perp E$. Now everything follows from King's criterion (Lemma \ref{L:King}).
\end{proof}

\begin{remark} In general, $\sigma_\beta$-unstable points in $\Rep_\alpha(\Perp E)$ may become $\sigma_{\beta_E^\vee}$-semi-stable under $\R_{Q_E}$. This is equivalent to say that $\sigma_{\beta_E^\vee}$-semi-stable points may become unstable under $\iota_E$. The following lemma is an easy consequence of King's criterion.
\end{remark}

\begin{lemma}
A $\sigma_\beta$-unstable representation $M\in \Rep_\alpha(\Perp E)$ is $E$-effective
if and only if there is no subrepresentation $($resp. quotient$)$ $M_1$ of $M$ such that $M_1\in\Perp E$ and $\sigma_\beta(\dim_QM_1)>0$ $($resp. $\sigma_\beta(\dim_QM_1)<0)$.
\end{lemma}

Now we turn to the inverse functor $\iota_E$. Let $p$ be the projection: $$\Rep_\alpha(\Perp E)\cong\GL_\alpha\times_{\GL_{\alpha_\epsilon}}\Rep_{\alpha_\epsilon}(Q_E)\to\Rep_{\alpha_\epsilon}(Q_E)$$ and $\Rep_{\alpha_\epsilon}\ss{\beta}{ss}(Q_E)$ be the open set $p(\Rep_\alpha\ss{\beta}{ss}(\Perp E))$ in $\Rep_{\alpha_\epsilon}(Q_E)$. By definition, it consists exactly of all representations $M$ such that $\iota_E(M)$ is $\sigma_\beta$-semi-stable. We also denote $p(\Rep_\alpha\ss{\beta}{st}(\Perp E))$ and $p(\Rep_\alpha\ss{\beta_E^\vee}{us}(\Perp E))$ by $\Rep_{\alpha_\epsilon}\ss{\beta}{st}(Q_E)$ and $\Rep_{\alpha_\epsilon}\ss{\beta}{su}(Q_E)$ respectively. Note that $\Rep_{\alpha_\epsilon}\ss{\beta}{su}(Q_E)$ consists exactly all representations $M\in\Rep_{\alpha_\epsilon}\ss{\beta_E^\vee}{ss}(Q_E)$ such that $\iota_E(M)$ is $\sigma_\beta$-unstable, so it is closed in $\Rep_{\alpha_\epsilon}\ss{\beta_E^\vee}{ss}(Q_E)$.

We give a useful criterion on when $\Rep_{\alpha_\epsilon}\ss{\beta}{ss}(Q_E)=\Rep_{\alpha_\epsilon}\ss{\beta_E^\vee}{ss}(Q_E)$, i.e., $\iota_E$ preserves (semi-)stable points. At certain points below, we may assume that $\sigma_\beta$ is a {\em $E^\vee$-regular weight}, that is, $\beta_E^\vee=\beta_\epsilon^\vee$. Equivalently, for a general $N\in\Rep_{\beta}(Q)$, the universal homomorphism $\tau^{-1}(hE)\to N$ is injective. Thus $(n\beta)_E^\vee=n\beta_\epsilon^\vee$ holds for all $n\in\mb{N}$.

\begin{remark} We want to point out this is really not a big assumption.
As an easy consequence of Generalized Fulton's Conjecture \cite[Theorem 7.16]{DW2}, $\alpha$ is strictly $\sigma_{\tau^{-1}\epsilon}$-semi-stable, so $\tau^{-1}\epsilon$ must lie on the boundary of the cone $\Sigma_\alpha(Q)$. In fact, in many cases it is an extremal ray of $\Sigma_\alpha(Q)$, then $\tau^{-1}(hE)$ must inject into $N$.
\end{remark}

\begin{proposition} \label{P:sspreserve} Assume that $\sigma_\beta$ is a $E^\vee$-regular weight, then
$\GL_{\beta_\epsilon^\vee}\cdot\R_{Q_E}^\vee(\Rep_\beta(Q))$ contains a dense subset of $\Rep_{\beta_\epsilon^\vee}(Q_E)$ if and only if \begin{enumerate}
\item[(i)] $\innerprod{\beta,\epsilon}_Q\leqslant 0$, or
\item[(ii)] $\innerprod{\beta,\epsilon}_Q>0$ and $\ext_Q(\beta,\tau^{-1}\epsilon)=0$.
\end{enumerate}
In these cases, $\iota_E$ preserves semi-stable points.
\end{proposition}

\begin{proof} Suppose that $-m=\innerprod{\beta,\epsilon}_Q\leqslant 0$. Let $M$ be any representation in $\Rep_{\tilde{\beta}_\epsilon^\vee}(\Perp E)$, then $M\oplus m(\tau^{-1} E)$ has dimension $\beta=\tilde{\beta}_\epsilon^\vee+m(\tau^{-1}\epsilon)$. We have $\tilde{\beta}_\epsilon^\vee\perp E$ because $\sigma_\beta$ is $E^\vee$-regular. Clearly $\R_E^\vee(M\oplus m(\tau^{-1} E))=M$, so $\GL_{\beta_\epsilon^\vee}\cdot\R_{Q_E}^\vee(\Rep_\beta(Q))$ contains $\Rep_{\beta_\epsilon^\vee}(Q_E)$.

If $\innerprod{\beta,\epsilon}_Q>0$, then $\ext_Q(\tau^{-1}\epsilon,\beta)>0$.
Let $\beta'$ be the dimension vector of $M'$ in the universal homomorphism sequence $0\to h(\tau^{-1}E) \to M \to M'\to 0$. By the dual construction of Lemma \ref{L:OP} and Lemma \ref{L:codimgen},
$\ext_Q(\beta',\tau^{-1}\epsilon)=0$ is a necessary and sufficient condition for $\GL_{\beta_\epsilon^\vee}\cdot\R_{Q_E}^\vee(\Rep_\beta(Q))$ containing a dense subset of $\Rep_{\beta_\epsilon^\vee}(Q_E)$. It follows from the sequence that $\ext_Q(\beta',\tau^{-1}\epsilon)=0$ if and only if $\ext_Q(\beta,\tau^{-1}\epsilon)=0$.

Now if $M\in\Rep_{\alpha_\epsilon}(Q_E)$ is $\sigma_{\beta_\epsilon^\vee}$-semi-stable, then there is some general representation $N\in\Rep_{n\beta_\epsilon^\vee}(Q_E)$ such that $N\perp M$. Since $\sigma_\beta$ is a $E^\vee$-regular weight, we can assume that $N=\R_{Q_E}^\vee(\tilde{N})$ for some $\tilde{N}\in\Rep_{n\beta}(Q)$. By the adjointness, $\Hom_Q(\tilde{N},\iota_E(M))=\Hom_{Q_E}(N,M)=0$. Since $\sigma_\beta(\alpha)=0$, $\tilde{N}\perp\iota_E(M)$ and thus $\iota_E(M)$ is $\sigma_\beta$-semi-stable.
\end{proof}

\begin{definition} A weight $\sigma_\beta$ is called {\em weakly $\mc{E}_\alpha$-effective} if $\mc{E}_\alpha\cap\Rep_\alpha\ss{\beta_E^\vee}{ss}(Q)$ is non-empty. It is called {\em $\mc{E}_\alpha$-effective} if in addition $\innerprod{\beta,\epsilon}_Q>0$.
\end{definition}

Note that for any $M\in\mc{E}_\alpha$, $E$ is a quotient representation of $M$. So if $\sigma_\beta$ is weakly $\mc{E}_\alpha$-effective, then at least we have that $\innerprod{\beta,\epsilon}_Q\geqslant 0$. If $\mc{E}_\alpha\cap\Rep_\alpha\ss{\beta_E^\vee}{ss}(Q)$ contains a $\sigma_\beta$-stable point, then $\innerprod{\beta,\epsilon}_Q>0$. In this case, $\mc{E}_\alpha\cap\Rep_\alpha\ss{\beta_E^\vee}{ss}(Q)$ has codimension one in $\Rep_\alpha(Q)$, so its closure descends to a divisor in the corresponding GIT quotient.
It is possible that there exists weakly $\mc{E}_\alpha$-effective weight but no $\mc{E}_\alpha$-effective weight, see Example \ref{Ex:2c}.

\begin{align*}\text{Let }\mc{E}_\alpha^{\sigma_\beta}:&=\{M\in \mc{E}_\alpha\cap\Rep_\alpha\ss{\beta}{ss}(Q)\mid \R_E(M)\in \Mod_\alpha\ss{\beta_E^\vee}{us}(\Perp E)\} \\
& =\{M\in \mc{E}_\alpha\cap\Rep_\alpha\ss{\beta_E^\vee}{ss}(Q)\mid \R_E(M)\in \Mod_\alpha\ss{\beta}{un}(Q)\}.
\end{align*}
So $p(\mc{E}_\alpha^{\sigma_\beta})=p(\Rep_\alpha\ss{\beta}{un}(\Perp E))\cap\Rep_{\alpha_\epsilon}\ss{\beta_E^\vee}{ss}(Q_E)$ is closed in $\Rep_{\alpha_\epsilon}\ss{\beta_E^\vee}{ss}(Q_E)$.

\begin{lemma} \label{L:positive} {\ } \begin{enumerate}
\item[(i)] If $\innerprod{\beta,\epsilon}_Q>0$, then $M\in\mc{E}_\alpha$ implies $\R_E(M)$ is $\sigma_\beta$-unstable. In particular, $\mc{E}_\alpha^{\sigma_\beta}=\mc{E}_\alpha\cap\Rep_\alpha\ss{\beta_E^\vee}{ss}(Q)$.
\item[(ii)]  If $\sigma_\beta$ is $E^\vee$-regular and $\mc{E}_\alpha^{\sigma_\beta}$ is non-empty, then $\innerprod{\beta,\epsilon}_Q>0$. \end{enumerate}
Hence, the additional condition $\innerprod{\beta,\epsilon}_Q>0$ can be replaced by $\mc{E}_\alpha^{\sigma_\beta}=\mc{E}_\alpha\cap\Rep_\alpha\ss{\beta_E^\vee}{ss}(Q)$ or $\mc{E}_\alpha^{\sigma_\beta}$ is non-empty when $\sigma_\beta$ is $E^\vee$-regular.
\end{lemma}

\begin{proof} (i) Let us look at the second construction of Lemma \ref{L:OP}. Since $hE$ is a quotient representation of $M$ and $\sigma_\beta(\epsilon)>0$, $\sigma_\beta(\dim_Q(M''))<0$. But $M''$ is a subrepresentation of $\R_E(M)$, so $\R_E(M)$ is $\sigma_\beta$-unstable.\\
(ii) Pick a representation $M\in\mc{E}_\alpha^{\sigma_\beta}$, then by definition and Lemma \ref{L:pi_E_ss}, $\R_E(M)$ is $\sigma_{\tau^{-1}\epsilon}$ and $\sigma_{\tilde{\beta}_\epsilon^\vee}$ semi-stable but $\sigma_\beta$-unstable. So there is a quotient representation $N$ of $\R_E(M)$ with $\dim N=\gamma$ such that $\sigma_\beta(\gamma)<0$. Then $\sigma_{\tilde{\beta}_\epsilon^\vee}(\gamma)=\sigma_\beta(\gamma)+\innerprod{\beta,\epsilon}_Q\sigma_{\tau^{-1}\epsilon}(\gamma)\geqslant 0$ implies that $\innerprod{\beta,\epsilon}_Q>0$.
\end{proof}

Let $q:\Rep_\alpha\ss{\beta}{ss}(Q)\to \Mod_\alpha^{\sigma_\beta}(Q)$ be the GIT quotient maps. This map is the composition of the stack quotient $\hat{q}:\Rep_\alpha\ss{\beta}{ss}(Q)\to \Mod_\alpha\ss{\beta}{ss}(Q)$  and the natural morphism $\tilde{q}: \Mod_\alpha\ss{\beta}{ss}(Q)\to\Mod_\alpha^{\sigma_\beta}(Q)$ \cite{G}.
We use the similar notation for the GIT quotient $q_E: \Rep_{\alpha_\epsilon}\ss{\beta_E^\vee}{ss}(Q_E)\to\Mod_{\alpha_\epsilon}^{\sigma_{\beta_E^\vee}}(Q_E)$.

\begin{theorem} \label{T:birational}
The functor $\R_{Q_E}$ induces a birational transformation $$\varphi_E:\Mod_\alpha^{\sigma_\beta}(Q)\dashrightarrow\Mod_{\alpha_\epsilon}^{\sigma_{\beta_E^\vee}}(Q_E),$$
whose fundamental set lies outside $q(\Rep_\alpha\ss{\beta_E^\vee}{ss}(Q))$.
It maps $q(\Rep_\alpha\ss{\beta}{ss}(\Perp E))$ isomorphically onto $q_E(\Rep_{\alpha_\epsilon}\ss{\beta}{ss}(Q_E))$ and maps $q(\overline{\mc{E}_\alpha^{\sigma_\beta}})$ onto the closed set $\tilde{q}_E\R_{Q_E}(\mc{E}_\alpha^{\sigma_\beta})$.
\end{theorem}

\begin{proof} By Lemma \ref{L:morphism}, we have already got a morphism from $\Rep_\alpha(E^\reg)$ to the moduli stack $\Mod_{\alpha_\epsilon}(Q_E)$.  By definition this morphism descends to $\Rep_\alpha\ss{\beta_E^\vee}{ss}(Q) \to\Mod_{\alpha_\epsilon}^{\sigma_{\beta_E^\vee}}(Q_E)$. It is clearly constant on the orbit, so we use the fact that $q$ is a categorical quotient to obtain a morphism $\varphi: q(\Rep_\alpha\ss{\beta_E^\vee}{ss}(Q))\to\Mod_{\alpha_\epsilon}^{\sigma_{\beta_E^\vee}}(Q_E)$. Apply a similar argument to the functor $\iota_E$, we get another morphism $q_E( \Rep_{\alpha_\epsilon}\ss{\beta}{ss}(Q_E) )\to\Mod_\alpha^{\sigma_\beta}(Q)$. This is a rational inverse of $\varphi_E$ because the property of $\iota_E$ ensure it maps $q_E( \Rep_{\alpha_\epsilon}\ss{\beta}{ss}(Q_E))$ isomorphically onto its image, which is $q(\Rep_\alpha\ss{\beta}{ss}(\Perp E))$ by Lemma \ref{L:homofibre}. Clearly, $\varphi_E$ maps $q(\overline{\mc{E}_\alpha^{\sigma_\beta}})$ onto $\tilde{q}_E\R_{Q_E}(\mc{E}_\alpha^{\sigma_\beta})$, which is closed.

The situation is depicted in the following diagram:$$\xymatrix @C=2.0pc {
\Rep_\alpha\ss{\beta}{ss}(Q) \ar@{->>}^{q}[r]  & \Mod_\alpha^{\sigma_\beta}(Q)  \ar@{.>}[d]^{\varphi_E} & \ar[l] \Rep_{\alpha_\epsilon}\ss{\beta}{ss}(Q_E) \ar@{_{(}->}[d] \\
\Rep_\alpha\ss{\beta_E^\vee}{ss}(Q) \ar[r] \ar@{_{(}->}[u]  & \Mod_{\alpha_\epsilon}^{\sigma_{\beta_E^\vee}}(Q_E)  \ar@<1ex>@{.>}[u]^{\varphi_E^{-1}} & \ar@{->>}[l]_{q_E} \Rep_{\alpha_\epsilon}\ss{\beta_E^\vee}{ss}(Q_E)\\
}$$
\end{proof}

\begin{remark} Recall that we have the following decompositions \begin{align*}
&\Rep_\alpha\ss{\beta_E^\vee}{ss}(Q)=\Rep_\alpha\ss{\beta}{ss}(\Perp E)\amalg(\mc{E}_\alpha\cap\Rep_\alpha\ss{\beta_E^\vee}{ss}(Q));\\
&\Rep_{\alpha_\epsilon}^{\sigma_{\beta_E^\vee}}(Q_E)=\Rep_{\alpha_\epsilon}\ss{\beta}{ss}(Q_E)\amalg \Rep_{\alpha_\epsilon}\ss{\beta}{su}(Q_E).\end{align*}
So if
\begin{equation} \label{eq:morphism} 
\Mod_\alpha^{\sigma_\beta}(Q)\setminus q(\Rep_\alpha\ss{\beta}{ss}(\Perp E))\subseteq q(\mc{E}_\alpha\cap\Rep_\alpha\ss{\beta_E^\vee}{ss}(Q)),
\end{equation}
then $\varphi_E$ is a morphism.
If in addition the following two equivalent conditions are satisfied
\begin{equation} \label{eq:contraction} q(\Rep_\alpha\ss{\beta_E^\vee}{us}(Q))\subseteq \tilde{q}\R_E(\mc{E}_\alpha^{\sigma_\beta})\ \text{ or }\ q_E(\Rep_{\alpha_\epsilon}\ss{\beta}{su}(Q_E))\subseteq \tilde{q}_E\R_{Q_E}(\mc{E}_\alpha^{\sigma_\beta}), \end{equation}
then $\varphi_E$ is surjective, and only contracts $q(\overline{\mc{E}_\alpha^{\sigma_\beta}})$ to $\tilde{q}_E\R_{Q_E}(\mc{E}_\alpha^{\sigma_\beta})$.
The contraction may not be strict in the sense that the generic fibre can be zero-dimensional. In this generality, we cannot say too much on this contraction. However, we will see in the next section that the essential case is the blow-up of $\Mod_{\alpha_\epsilon}^{\sigma_{\beta_E^\vee}}(Q_E)$ along $\tilde{q}_E\R_{Q_E}(\mc{E}_\alpha^{\sigma_\beta})$.
\end{remark}

Recall our notation that $\beta^\epsilon=\iota_\epsilon(\beta)$, so $\R_{Q_\epsilon}(\beta^\epsilon)=\beta$.
\begin{corollary} \label{C:iso} In the situation $\text{(i) or (ii)}$ of Proposition \ref{P:sspreserve}, the birational transformation $\varphi_E$ is an isomorphism.
In particular, if $\alpha,\beta$ are dimension vectors of $Q_E$ such that $\alpha^\epsilon\perp E$ and $\beta^\epsilon\perp E$, then $\Mod_\alpha^{\sigma_\beta}(Q_E)\cong\Mod_{\alpha^\epsilon}^{\sigma_{\beta^\epsilon}}(Q)$.
\end{corollary}


Motivated by Lemma \ref{L:positive}, Theorem \ref{T:birational}, and Corollary \ref{C:iso}, we make the following definition:
\begin{definition} The {\em core} $\Sigma_\alpha^\heartsuit(Q)$ of the $G$-ample cone $\Sigma_\alpha(Q)$ is the subcone defined by $\innerprod{-,\epsilon}_Q>0$ for all real roots $\epsilon$ right orthogonal to $\alpha$; and
$\innerprod{\epsilon,-}_Q<0$ for all real roots $\epsilon$ left orthogonal to $\alpha$. Its boundary $\partial\Sigma_\alpha^\heartsuit(Q)$ is called the {\em shell} of the core. In practice, a more useful definition is the {\em weak core} if we add the restriction $\epsilon<\alpha$.
\end{definition}


For a weight $\sigma_\beta$ on the boundary $\partial\Sigma_\alpha^\heartsuit(Q)$, there usually exists strictly $\sigma_\beta$-semi-stable points. If there is no strictly $\sigma_\beta$-semi-stable points, then by Lemma \ref{L:homofibre} and Proposition \ref{P:sspreserve} the whole $C_{\tau^{-1} E}$ consists of unstable points. So the {\em null-cone} has codimension one, which is quite rare considering $\alpha$ is a non-isotropic imaginary Schur root. So basically the boundary consists of a subset of {\em walls} in the sense of variational GIT theory \cite[Definition 3.3.1]{DH}. Here we give it a special name ``shell" to stress its importance in the quiver setting. From the view of Corollary \ref{C:iso}, the birational transformation $\varphi_E$ is just a special type of {\em wall-crossing}.

\begin{lemma} \label{L:acproj} If the anti-canonical weight $\sigma_{ac}=\innerprod{\alpha+\tau^{-1}\alpha,-}$ is $E^\vee$-regular, then $\sigma_{ac_\epsilon^\vee}$ is the anti-canonical weight for $\Rep_{\alpha_\epsilon}(Q_E)$.
\end{lemma}
\begin{proof} We need to verify that $\R_{Q_\epsilon}^\vee(\alpha+\tau^{-1}\alpha)=\alpha_\epsilon+\tau_{Q_E}^{-1}(\alpha_\epsilon)$. But this is clear from the formula $\tau_{Q_E}^{-1}=\R_{Q_\epsilon}^\vee \tau^{-1}\iota_\epsilon$.
\end{proof}

\begin{example} \label{Ex:1c} (Example \ref{Ex:1} continued)

Let us consider the anti-canonical weight for $\Rep_\alpha(B^1)$, i.e., $\sigma_\beta=(2,1,-3)$, then $\sigma_\beta$ is $\mc{E}_\alpha$-effective.
One can verify that $\tau^{-1}\epsilon=(1,0,1)$ is an extremal ray of the cone $\Sigma_\alpha(B^1)$, so $\sigma_\beta$ is a $E^\vee$-regular weight and $\sigma_{\beta_\epsilon^\vee}=(3,-3)$.
It is clear that $\Mod_{\alpha_\epsilon}^{\sigma_{\beta_\epsilon^\vee}}(\Theta_3)$ is the projective plane $\mb{P}^2$. Since there is no strictly semi-stable points in $\Rep_\alpha(B^1)$, $\Mod_\alpha^{\sigma_\beta}(B^1)$ is smooth. Elements in $\Rep_\alpha\ss{\beta_\epsilon^\vee}{us}(B^1)$ satisfy that $C_1=D_1=0,B_1\neq 0$, so $\R_{Q_E}(\Rep_\alpha\ss{\beta_\epsilon^\vee}{us}(B^1))$ is a single point $N_1$ represented by $\{A=B=0; C=1\}$, which is exactly $\R_{Q_E}(\mc{E}_\alpha^{\sigma_\beta})$. Now by Theorem \ref{T:birational} and its remark, the birational morphism $\varphi_E: \Mod_\alpha^{\sigma_\beta}(B^1)\to \Mod_{\alpha_\epsilon}^{\sigma_{\beta_\epsilon^\vee}}(\Theta_3)$ contracts the curve $q(\mc{E}_\alpha^{\sigma_\beta})$ to a point $N_1$. So by the Castelnuovo's criterion (\cite[Theorem 5.7]{Ha}) the curve must be a $-1$-curve and $\varphi_E$ is the blow-up at $N_1$. This can also be seen from Theorem \ref{T:blow-up} later. It is clear from Corollary \ref{C:iso} that the chambers of $\tau\Sigma_\alpha(B^3)$ can be described as follows: the core in red gives the blow-up at one point while the dark part gives $\mb{P}^2$.

\begin{center}
\includegraphics[width=4in]{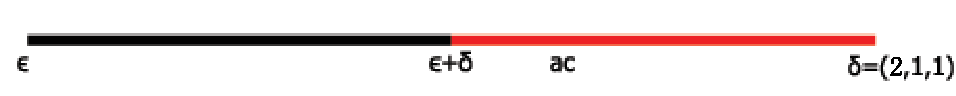}
\end{center}

Let us consider the anti-canonical weight for $\Rep_\alpha(B^2)$, i.e., $\sigma_\beta=(2,2,-1,-3)$, then $\sigma_\beta$ is $\mc{E}_\alpha$-effective for $E=E_2$. One can verify that $\tau^{-1}\epsilon=(0,1,0,1)$ is an extremal ray of the cone $\Sigma_\alpha(B^2)$, so $\sigma_\beta$ is a $E^\vee$-regular weight and $\sigma_{\beta_\epsilon^\vee}=(2,1,-3)$. For the same reason, $\Mod_\alpha^{\sigma_\beta}(B^2)$ is smooth. Elements in $\Rep_\alpha\ss{\beta_\epsilon^\vee}{us}(B^2)$ satisfy that $B_2=D_2=0$ and $A_2C_2E_2\neq 0$, so $\R_E(\Rep_\alpha\ss{\beta_\epsilon^\vee}{us}(B^2))$ is a single point $N_2$ represented by $\{B_1=D_1=0; A_1=C_1=1\}$, which is exactly $\R_E(\mc{E}_\alpha^{\sigma_\beta})$. By Theorem \ref{T:birational} and its remark, the birational morphism $\varphi_E: \Mod_\alpha^{\sigma_\beta}(B^2)\to \Mod_{\alpha_\epsilon}^{\sigma_{\beta_\epsilon^\vee}}(B^1)$ is the blow-up at $N_2$.  Note that $\varphi_E(N_2)\in \Perp E$ of $B^1$, so it does not lie on the $-1$-curve of $\Mod_\alpha^{\sigma_\beta}(B^1)$. Hence $\Mod_\alpha^{\sigma_\beta}(B^2)$ is the blow-up of $\mb{P}^2$ at two general points. We leave to interested readers to check that for $E=E_3$, $\varphi_E: \Mod_\alpha^{\sigma_\beta}(B^2)\to \Mod_{\alpha_\epsilon}^{\sigma_{\beta_\epsilon}}(\Theta_{2,2})$ is the blow-up of $\mb{P}^1\times\mb{P}^1$ at one point. Using Corollary \ref{C:iso}, one can easily verify that the chambers of $\tau\Sigma_\alpha(B^3)$ can be described as follows: the red core, and yellow, blue, green parts gives the blow-up at two general points, the blow-up at one point, $\mb{P}^1\times\mb{P}^1$, and $\mb{P}^2$ respectively.

\begin{center}
\includegraphics[width=4in]{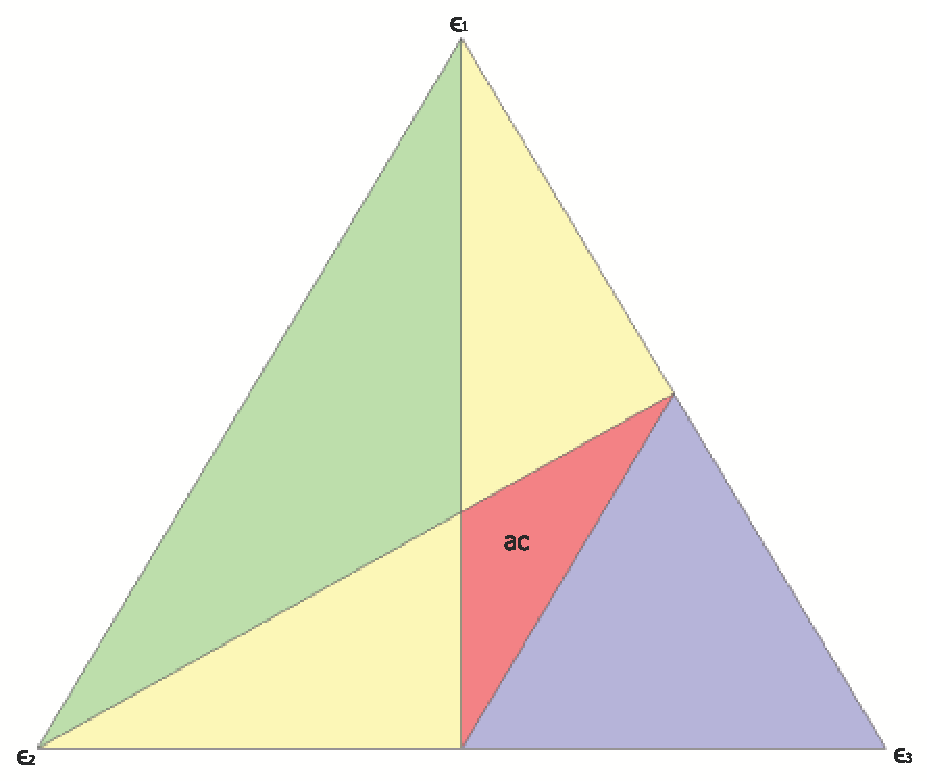}
\end{center}

Let us consider the anti-canonical weight for $\Rep_\alpha(B^3)$, i.e., $\sigma_\beta=(2,2,2,-3,-3)$, then $\mc{E}_\alpha$-effective for $E=E_{35}$. One can verify that $\tau^{-1}\epsilon=(0,0,1,1,0)$ is an extremal ray of the cone $\Sigma_\alpha(B^3)$, so $\sigma_\beta$ is a $E^\vee$-regular weight and $\sigma_{\beta_\epsilon^\vee}=(2,2,-1,-3)$. For the same reason, $\Mod_\alpha^{\sigma_\beta}(B^3)$ is smooth. Elements in $\Rep_\alpha\ss{\beta_\epsilon^\vee}{us}(B^3)$ satisfy that $B_3=D_3=0$ and $A_3C_3E_3F_3\neq 0$, so $\R_E(\Rep_\alpha\ss{\beta_\epsilon^\vee}{us}(B^3))$ is a single point $N_3$ represented by $\{B_2=D_2=0; A_2=C_2=E_2=1\}$, which is exactly $\R_E(\mc{E}_\alpha^{\sigma_\beta})$. By Theorem \ref{T:birational} and its remark, the birational morphism $\varphi_E: \Mod_\alpha^{\sigma_\beta}(B^3)\to \Mod_{\alpha_\epsilon}^{\sigma_{\beta_\epsilon^\vee}}(B^2)$ is the blow-up at $N_3$.  Note that $\varphi_E(N_3)\in\Perp E_3$ of $B^2$, so it does not lie on the $-1$-curve passing the two blow-up points. Hence $\Mod_\alpha^{\sigma_\beta}(B^3)$ is the blow-up of $\mb{P}^2$ at three general points. Using Corollary \ref{C:iso}, one can easily verify that the chambers of $\tau\Sigma_\alpha(B^3)$ can be described as follows:
the red-hexahedron core gives the blow-up at three general points, the six tetrahedrons all give the blow-up at two general points, and the rest gives either the blow-up at one point, $\mb{P}^1\times\mb{P}^1$, or $\mb{P}^2$.

\begin{center}
\includegraphics[width=4in]{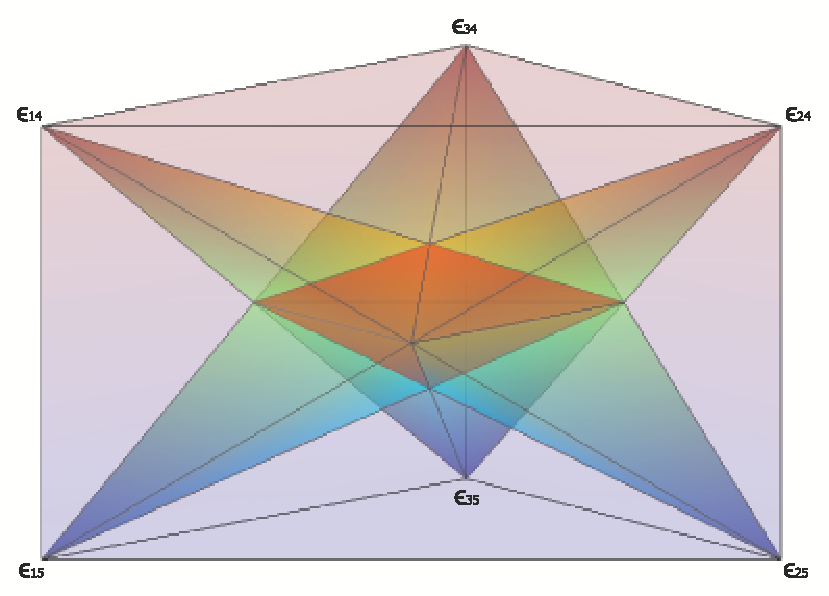}
\end{center}

The above examples are well-known. The earliest reference that we found is \cite{Hi}.

Let us consider the anti-canonical weight for $\Rep_\alpha(B^4)$, i.e., $\sigma_\beta=(2,2,2,2,2,-5)$, then $\sigma_\beta$ is $\mc{E}_\alpha$-effective. One can verify that $\tau^{-1}\epsilon=(0,0,0,1,1,1)$ is an extremal ray of the cone $\Sigma_\alpha(B^4)$, so $\sigma_\beta$ is a $E^\vee$-regular weight and $\sigma_{\beta_\epsilon^\vee}=(2,2,2,-3,-3)$. For the same reason, $\Mod_\alpha^{\sigma_\beta}(B^4)$ is smooth. Elements in $\Rep_\alpha\ss{\beta_\epsilon^\vee}{us}(Q)$ satisfy that $\rank(A_4,B_4,C_4)=1$, so $\R_E(\Rep_\alpha\ss{\beta_\epsilon^\vee}{us}(Q))$ is a single point $N_4$ represented by $\{A_3=B_3=C_3=D_3=E_3=F_3=1\}$, which is exactly $\R_E(\mc{E}_\alpha^{\sigma_\beta})$. By Theorem \ref{T:birational} and its remark, the birational transformation $\varphi_E: \Mod_\alpha^{\sigma_\beta}(B^4)\to \Mod_{\alpha_\epsilon}^{\sigma_{\beta_\epsilon}}(B^3)$ is the blow-up at $N_4$.  Note that $\varphi_E(N_4)\in \Perp E_{uv}$ of $B^3$ for all $u,v$, so it does not lie on any $-1$-curve passing two blow-up points. Hence $\Mod_\alpha^{\sigma_\beta}(B^4)$ is the blow-up of $\mb{P}^2$ at four general points. The chambers of $\tau\Sigma_\alpha(B^4)$ has a similar structure.

\end{example}

If $\Sigma_\alpha^\heartsuit(Q)$ is empty, then according to Corollary \ref{C:iso}, for almost every weight $\sigma_\beta$, we can find an exceptional $E$ such that $\varphi_E$ is an isomorphism. In this sense, we say that

\begin{definition} $(Q,\alpha)$ is {\em weakly reduced} if $\Sigma_\alpha^\heartsuit(Q)$ is non-empty. $(Q,\alpha)$ is {\em reduced} if there is a weight $\sigma_\beta$ such that $\varphi_E$ is not an isomorphism for any exceptional representation $E$. $(Q,\alpha)$ is {\em strongly reduced} if there is a weight $\sigma_\beta$ such that $\Mod_\alpha^{\sigma_\beta}(Q)$ is not a moduli for any quiver having vertices less than that of $Q$. We also call the above triple $(Q,\alpha,\sigma_\beta)$ reduced and strongly reduced. $(Q,\alpha)$ is {\em minimal} if its core is the same as its cone.
\end{definition}

It is possible that $(Q,\alpha)$ is weakly reduced but not reduced, see Example \ref{Ex:0c}. In this case, all possible GIT quotients inside the core already appear on the shell.
However, we do not know if strongly reduced is really stronger than reduced.

\begin{conjecture} ``Strongly reduced" and ``reduced" are equivalent.
\end{conjecture}

\begin{conjecture} If $(Q,\alpha)$ is reduced, then the weak core coincides with the core.
\end{conjecture}

Starting with a triple $(Q,\alpha,\sigma_\beta)$ as in the beginning of this section, the first step to study the moduli space $\Mod_\alpha^{\sigma_\beta}(Q)$ is to make the triple reduced. This can be done as follows. Find a real root $\epsilon$ orthogonal to $\alpha$, if $\beta$ falls into the two situation of Proposition \ref{P:sspreserve}, then we project it. Although the second criterion is sharp in certain cases (Example \ref{Ex:1c}), it is just a sufficient condition in general. In practice, it may be not as useful as elementary arithmetic consideration. However, the main difficulty is that we do not know a good algorithm to detect real roots orthogonal to $\alpha$.

\begin{example} \label{Ex:2c} (Example \ref{Ex:2} continued)

What happened if we take the anti-canonical weight again for $\Rep_\alpha(B^{4,1})$? In this case, $\sigma_\beta=(2,2,2,2,2,0,-5)$, then it has strictly semi-stable points this time. One can verify that $\ext_Q(\beta,\tau^{-1}\epsilon_1)=0$, so by Corollary \ref{C:iso}, $\Mod_\alpha^{\sigma_\beta}(B^{4,1})$ is still the blow-up of $\mb{P}^2$ at four general points. We leave it for interested readers to verify that the core is empty, i.e., $(Q,\alpha)$ is not reduced, so the only GIT quotients for $\Rep_\alpha(B^{4,1})$ are the blow-up of $n$ points of $\mb{P}^2$, where $n\leqslant 4$, and $\mb{P}^1\times\mb{P}^1$.
\begin{conjecture}
Those are all the possible 2-dimensional GIT quotients in the quiver setting.
\end{conjecture}
\end{example}

\begin{example} \label{Ex:0c} (Example \ref{Ex:0} continued) The cone $\Sigma_\alpha(Q)$ is generated by two extremal rays $(0,1,1)$ and $(2,2,3)$. With a little effort, one can show that the core $\Sigma_\alpha^\heartsuit(Q)$ is cut out by a single real root $(0,2,1)$. So the core is generated by two extremal rays $(0,1,1)$ and $(2,3,4)$. But an elementary arithmetic argument can show that the quotient is constantly $\mb{P}^2$.
\end{example}

\begin{example} \label{Ex:SB}
It is possible that $\varphi_E$ is not a morphism, i.e., strictly birational. Take any weight in the green area of $\Sigma_\alpha(B^2)$, then we see from the picture that $\varphi_E$ induces a birational transformation $\mb{P}^2 \to \mb{P}^1\times\mb{P}^1$, which is strict.
\end{example}

One can blame this example for the birational transformation crosses too many walls. In fact, we will see in the next section that crossing a part of shell which does not intersect any other wall behaves quite agreeable and can be described explicitly.

\section{The Shell-crossing} \label{S:SC}

A fundamental problem in the variational GIT is to describe how the quotients change when the weight crosses a single wall. As shown in the example \ref{Ex:SB}, the birational transformation $\varphi_E$ constructed in Theorem \ref{T:birational} can cross several walls. However, we can always appropriately choose weights such that a {\em piece} of the shell is the only wall being crossed.

Remarkably, Michael Thaddeus gave an beautiful solution to the above problem in \cite{T}. Let us recall some of his general results. Let $G$ be a reductive group over $k$ acting on a quasi-projective variety over $k$. Suppose that $L_+$ and $L_-$ are two ample $G$-linearization such that if $L(t)=L_+^t L_-^{1-t}$ for $t\in[-1,1]$, there exists $t_0\in(-1,1)$ such that $X^{\rm ss}(t)=X^{\rm ss}(+)$ for $t>t_0$ and $X^{\rm ss}(t)=X^{\rm ss}(-)$ for $t<t_0$. This certainly happens when we cross a wall only at $L(t_0)$ through the line segment $L_+L_-$ in the $G$-ample cone.
We denote
\begin{align*}
  X^{\pm} &:= X^{\rm ss}(\pm)\setminus X^{\rm ss}(\mp);\\
  X^0 &:= X^{\rm ss}(0)\setminus (X^{\rm ss}(+)\cup X^{\rm ss}(-)).
\end{align*}
It is easy to see that $X^{\pm}\subset X^0$ \cite[Lemma 4.1]{T}.

Let $x\in X^0$ be a smooth point such that $G\cdot x$ is closed in $X^{\rm ss}(0)$ and $G_x\cong k^*$. Since $x\in X^{\rm ss}(0)$, $G_x$ acts trivially on the fibre $(L_0)_x$. Assume that it acts non-trivially on $(L_+)_x$ with some negative weight $v_+$, then it acts on $(L_-)_x$ with some positive weight $v_-$. We require that two weights are coprime: $(v_+,v_-)=1$. It is shown that over a neighborhood of $x$ in $X^0\dslash{0}G$, $X^{\pm}\dslash{\pm}G$ are locally trivial fibrations with fibre the weighted projective space $\mb{P}(|\omega_i^{\pm}|)$. If furthermore all $\omega_i^{\pm}=\pm\omega$ for some $\omega$, then

\begin{lemma} \cite[Theorem 4.8]{T} \label{L:VGIT} Over a neighborhood of $x$ in $X^0\dslash{0}G$, $X^{\pm}\dslash{\pm}G$ are naturally isomorphic to the projective bundles $\mb{P}W^{\pm}$, their normal bundles are naturally isomorphic to $\pi_{\pm}^* W^{\mp}(-1)$, and the blow-ups of $X\dslash{\pm}G$ at $X^{\pm}\dslash{\pm}G$, and of $X\dslash{0}G$ at $X^0\dslash{0}G$, are all naturally isomorphic to the fibred product $X\dslash{-}G\times_{X\dslash{0}G}X\dslash{+}G$.
\end{lemma}

What is $X^{\pm}$ and $X^0$ in our setting? Let us first make one simple observation on the walls of $\Sigma_\alpha(Q)$. We write $H_\omega$ for the hyperplane defined by $\innerprod{-,\omega}_Q=0$, then the set of all walls is contained in the union of $H_\omega$ for all indivisible roots $\omega<\alpha$. Let $C^+,C^-$ be two adjacent chambers with $W$ being a common wall whose supporting hyperplane is given by $\innerprod{-,\omega}_Q=0$ for some indivisible root $\omega$. Here, we assume adjacency in a quite strong sense that $W$ has codimension one. In this case, the strictly semi-stable representations on $W$ must have a subrepresentation of dimension an integral point in the cone spanned by $\omega$ and $\alpha-\omega$. For simplicity, from now on let us assume that \begin{align*}
\alpha,\alpha-\omega \text{ are indivisible, and } 2\omega\nless\alpha,2(\alpha-\omega)\nless\alpha. \tag{$\smiley$}
\end{align*}
Note that with the assumption that $\alpha$ is indivisible, there is no strictly semi-stable points in $C^+$ and $C^-$, so in particular the corresponding quotients are smooth.

Let $\beta_{\pm}$ be an interior point of $C^{\pm}$, and $\beta_0$ be the intersection of $\beta_+\beta_-$ with $W$.
By definition, there is a strictly $\sigma_{\beta_0}$-semi-stable representation $M$ with a quotient representation $N$ of dimension $\gamma$ such that $\innerprod{\beta_0,\gamma}_Q=0$ but $\innerprod{\beta_+,\gamma}_Q\innerprod{\beta_-,\gamma}_Q<0$.
By assumption $\smiley$, we can assume $\gamma=\omega$ without loss of generality.
By our convention that $\innerprod{\beta_+,\omega}_Q>0$, $M$ is $\sigma_{\beta_+}$-stable but $\sigma_{\beta_-}$-unstable. We conclude that
\begin{align*}
  X^{+} &= \Rep_\alpha\ss{\beta_+}{ss}(Q)\cap\Rep_{\alpha\twoheadrightarrow\omega}(Q):=\Rep_{\alpha\twoheadrightarrow\omega}\ss{\beta_+}{ss}(Q),\\
  X^{-} &= \Rep_\alpha\ss{\beta_-}{ss}(Q)\cap\Rep_{\omega\hookrightarrow\alpha}(Q):=\Rep_{\omega\hookrightarrow\alpha}\ss{\beta_-}{ss}(Q).
\end{align*}
With a little effort, one can show that $X^0 = \Rep_\alpha\ss{\beta_0}{ss}(Q)\cap\Rep_{\omega\oplus\alpha-\omega}(Q)$, but we do not need this in the future.
Readers shouldn't find any difficulty to formulate these sets without assumption $\smiley$. Our main interest is the case when $\omega=\epsilon$ is a real Schur root (right) orthogonal to $\alpha$. Then $W=S$ is called a piece of the shell, and by our convention $C^+$ is inside the core. In view of the next lemma, it is usually an uninteresting case when $\epsilon$ is not exceptional to $\alpha$ (Definition \ref{D:exc}). So we assume that $\epsilon$ is exceptional to $\alpha$.

\begin{lemma} \label{L:pm} Assume that $\alpha-\epsilon$ and $\alpha$ are indivisible.
\begin{enumerate}
\item[(i)] If $\epsilon$ is orthogonal but not exceptional to $\alpha$, then both $\Rep_{\alpha\twoheadrightarrow\epsilon}\ss{\beta_+}{ss}(Q)$ and $\Rep_{\epsilon\hookrightarrow\alpha}\ss{\beta_-}{ss}(Q)$ are empty.
\item[(ii)] If $\epsilon$ is exceptional to $\alpha$, then
\begin{align*}\Rep_{\alpha\twoheadrightarrow\epsilon}\ss{\beta_+}{ss}(Q) & =\Rep_\alpha\ss{\beta_+}{ss}(Q)\cap\mc{E}_\alpha=\mc{E}_\alpha^{\sigma_\beta},\\
\Rep_{\epsilon\hookrightarrow\alpha}\ss{\beta_-}{ss}(Q) & =\Rep_\alpha\ss{\beta_-}{ss}(Q)\cap\GL_\alpha\cdot\R_E(\mc{E}_\alpha)= \GL_\alpha\cdot\R_E(\mc{E}_\alpha^{\sigma_\beta}).
\end{align*}
\end{enumerate}
\end{lemma}

\begin{proof} (i) Suppose that $\epsilon$ is not exceptional to $\alpha$ and there is a $\sigma_{\beta_+}$-semi-stable representation in the irreducible set $\Rep_{\alpha\twoheadrightarrow\epsilon}(Q)$, then general representation there is $\sigma_{\beta_+}$-semi-stable. We can assume that such a general representation $M$ has a subrepresentation $L$ general in $\Rep_{\alpha-\epsilon}(Q)$ with $M/L=E$. Then by Corollary \ref{C:exc}, $\hom_Q(L,E)>0$. But for any morphism in $\hom_Q(L,E)$, the only possible dimension for the image is either $\epsilon$, or a fraction of $\alpha-\epsilon$ or $\alpha$, because both $E$ and $L$ are semi-stable on the wall. The latter is ruled out by assumption. For the first case, we note that $\hom_Q(L/E,E)>0$ because $\mc{E}_\alpha$ is empty. So we can apply the similar argument inductively to $L/E$. The other half can be proved similarly.

(ii) The proof is similar to (i).
\end{proof}

Let us examine the assumption of the lemma in this setting.
Fortunately, smoothness is not an issue for us. $G\cdot x$ is closed simply means that the representation is {\em polystable}, i.e., a direct sum of stable representations. The condition $G_x\cong k^*$ is necessary as shown in \cite[Counterexample 5.8]{T}. In our setting, this condition is equivalent to that the polystable representation $M$ is a direct sum of two non-isomorphic stable representations. Then its stabilizer is a $2$-dimensional torus modulo the multi-diagonally embedded $k^*$. With assumption $\smiley$, the only possible dimensions of stable summands are $\epsilon$ and $\alpha-\epsilon$, but $2\epsilon\neq\alpha$, so we don't need to worry about this. Next, the coprime condition $(v_+,v_-)=1$ is not a problem as long as the crossing is general. Finally, the condition for weights $\epsilon_i$ is always satisfied by \cite[Proposition 4.9]{T}.


\begin{theorem} \label{T:blow-up} Suppose that assumption $\smiley$ holds, and $S$ is a piece of shell with supporting hyperplane $\innerprod{-,\epsilon}_Q$. If $\beta$ is $E^\vee$-regular and $S$ is the only wall intersecting $\beta\tilde{\beta}_\epsilon^\vee$, then
$\varphi_E:\Mod_\alpha^{\sigma_\beta}(Q)\to\Mod_{\alpha_\epsilon}^{\sigma_{\beta_\epsilon^\vee}}(Q_E)$
is the blow-up of $\Mod_{\alpha_\epsilon}^{\sigma_{\beta_\epsilon^\vee}}(Q_E)$ along the irreducible subvariety $\tilde{q}_E\R_{Q_E}(\Rep_{\epsilon\hookrightarrow\alpha}\ss{\tilde{\beta}_\epsilon^\vee}{ss}(Q))$. If the blow-up locus is non-empty, then it has dimension $-\innerprod{\epsilon,\alpha-\epsilon}_Q$ and its exceptional locus is $q(\Rep_{\alpha\twoheadrightarrow\epsilon}\ss{\beta}{ss}(Q))$.
\end{theorem}

\begin{proof} With the assumption $\smiley$, the condition of Lemma \ref{L:VGIT} holds for every $x\in X^0$, so its conclusion holds globally. Since the positive quotient is smooth and hence locally factorial, the codimension-one subvariety $\Rep_{\alpha\twoheadrightarrow\epsilon}\ss{\beta}{st}(Q)$ descends to a Cartier divisor in the positive quotient. So nothing happens after blowing it up. In the meanwhile, by Corollary \ref{C:iso} the negative quotient is isomorphic to $\Mod_{\alpha_\epsilon}^{\sigma_{\beta_\epsilon^\vee}}(Q_E)$ and under this identification $q(\Rep_{\epsilon\hookrightarrow\alpha}\ss{\tilde{\beta}_\epsilon^\vee}{ss}(Q))$ becomes $\tilde{q}_E\R_E(\Rep_{\epsilon\hookrightarrow\alpha}\ss{\tilde{\beta}_\epsilon^\vee}{ss}(Q))$. So our claim follows from Lemma \ref{L:VGIT}. For the statement on the dimension of the blow-up locus, it suffice to verify the assumption in Proposition \ref{P:dimblowup} holds. But if the assumption is not satisfied, it is clear from assumption $\smiley$ that the blow-up locus has to be empty.
\end{proof}

Without assumption $\smiley$, we can still say a lot about the birational morphism. For example, $\varphi_E$ is a Luna type stratification, whose local structure can be determined by the local quiver setting \cite{AL}, but we do not want to pursue this. At least, for all known {\em interesting} examples \cite{F2}, assumption $\smiley$ is always satisfied. To get interesting examples, we certainly hope that the blow-up is non-trivial. Roughly speaking, the number $-\innerprod{\epsilon,\alpha-\epsilon}_Q$ should lie in a proper range. When it is too small, for example zero, $S$ must be on the boundary of $\Sigma_\alpha(Q)$ if $S$ exists. When it is one, we again get a blow-up along a Cartier divisor. But when the number is too large, it is very likely that the set $\Rep_{\alpha\twoheadrightarrow\epsilon}(Q)$ is trapped into the null-cone. So in general for a $n$-dimensional quotient, we hope that $-\innerprod{\epsilon,\alpha-\epsilon}_Q$ is within the range $[2,n]$. However, this is by no means a sufficient condition. For example, one can verify that $-\innerprod{\epsilon_2,\alpha-\epsilon_2}_Q=2$ in Example \ref{Ex:2}, but the blow-up locus is empty. If the blow-up is non-trivial, then we say that piece of shell is {\em hard} in sense that one needs blow-up to cross it.

As pointed out in \cite{T}, there are basically two typical styles for single wall-crossings. One is that $X\dslash{-}G\to X\dslash{0}G$ is isomorphism while $X\dslash{+}G\to X\dslash{0}G$ is divisorial; the other is that both $X\dslash{-}G\to X\dslash{0}G$ and $X\dslash{+}G\to X\dslash{0}G$ are {\em small}, and usually this results a {\em flip} $X\dslash{-}G\dashrightarrow X\dslash{+}G$. According to Theorem \ref{T:blow-up}, a single shell-crossing almost always falls into the first category. Of course, there are other types of wall-crossings in the quiver setting. They seem to exclusively fall into the second category. For this type of examples, we refer the readers to \cite{F2}. We will pursue this in the follow-ups.

\begin{example} \label{Ex:3c} (Example \ref{Ex:3} continued) We take the weight $\sigma_\beta$ inside the core, which is generated by $(2,3,1)$ and $(3,4,3)$. It is easy to verify all the conditions of Theorem \ref{T:blow-up} are satisfied. The moduli downstairs $\Mod_{\alpha_\epsilon}^{\sigma_{\beta_\epsilon^\vee}}(Q_E)$ is clearly the projective space $\mb{P}^3$. The blow-up locus \eqref{eq:cubic} is thus a twisted cubic. Therefore, the moduli is the blow-up of $\mb{P}^3$ along a twisted cubic.
\end{example}

\begin{problem} If $\Mod_\alpha^{\sigma_\beta}(Q)$ is the blow-up of some smooth variety $X$ along an irreducible subvariety $Y$, can we always find a real root $\epsilon$ such that $X=\Mod_{\alpha_\epsilon}^{\sigma_{\beta_\epsilon^\vee}}(Q_E)$ and $Y=\tilde{q}_E\R_E(\Rep_{\epsilon\hookrightarrow\alpha}\ss{\tilde{\beta}_\epsilon^\vee}{ss}(Q))$ as in Theorem \ref{T:blow-up}?
\end{problem}

\section{Induced Ample Divisors} \label{S:IAD}

Let $q:\Rep_\alpha\ss{\beta}{ss}(Q)\to\Mod_\alpha^{\sigma_\beta}(Q)$ still be the quotient map. This is an equivariant proper map, so it maps $\GL_\alpha$-invariant closed sets to closed sets. For any $\GL_\alpha$-invariant divisor $C$ in $\Rep_\alpha\ss{\beta}{ss}(Q)$, if its support contains a $\sigma_\beta$-stable point, then it descends to a (Weil) divisor $q_*(C)$ in $\Mod_\alpha^{\sigma_\beta}(Q)$. Since we assumed that the characteristic of $k$ is $0$, this push-forward is nothing but take the invariants of the subscheme $C$ then throw away the components of codimension greater than one.
So for any $N\in\Rep_\gamma(Q)$, if the divisor $C_N$ contains a $\sigma_\beta$-stable point, then it descends to an effective divisor $D_N^{\sigma_\beta}:=q_*(C_N\cap\Rep_\alpha\ss{\beta}{ss}(Q))$ in $\Mod_\alpha^{\sigma_\beta}(Q)$. For what follows, whenever we write $D_N^{\sigma_\beta}$, we always assume it is a divisor. The next lemma can be deduced from the sequence \eqref{eq:canproj}.

\begin{lemma} \cite[Lemma 1]{DW1} \label{L:Dplus}
If $0\to N_1\to N\to N_2\to 0$ is an exact sequence with $\innerprod{\dim_Q(N_1),\alpha}=0$, then $c_N=c_{N_1}c_{N_2}$ and thus $D_N^{\sigma_\beta}=D_{N_1}^{\sigma_\beta}+D_{N_2}^{\sigma_\beta}$.
\end{lemma}

\begin{lemma} \label{L:Dequi}
Any divisor of form $D_N^{\sigma_\beta}$ are equivalent, where $N$ has a fixed dimension $\gamma$.
\end{lemma}

\begin{proof} Consider two such divisors $D_{N_1}^{\sigma_\beta}, D_{N_2}^{\sigma_\beta}$. Since the divisor is nontrivial, $\alpha$ is $\sigma_\gamma$-semi-stable, then $\alpha$ is $\sigma_{(n\beta-\gamma)}$-semi-stable for $n\gg 0$. So we can choose some $N_0\in\Rep_{(n\beta-\gamma)}(Q)$ such that $D_{N_0}^{\sigma_\beta}$ is non-trivial. Now consider the representation $N=N_1\oplus N_2\oplus N_0$. By Lemma \ref{L:Dplus}, $D_N^{\sigma_\beta}=D_{N_1}^{\sigma_\beta}+D_{N_2\oplus N_0}^{\sigma_\beta}=D_{N_2}^{\sigma_\beta}+D_{N_1\oplus N_0}^{\sigma_\beta}$. Since both $N_1\oplus N_0$ and $N_2\oplus N_0$ has dimension $n\beta$, the divisor $D_{N_1\oplus N_0}^{\sigma_\beta}-D_{N_2\oplus N_0}^{\sigma_\beta}$ corresponds to a rational function on $\Mod_\alpha^{\sigma_\beta}(Q)$. Hence, $D_{N_1}^{\sigma_\beta}$ and $D_{N_2}^{\sigma_\beta}$ are equivalent.
\end{proof}

We denote the class of $D_N^{\sigma_\beta}$ by $D_\gamma^{\sigma_\beta}$. We call it the {\em induced divisor from weight $\sigma_\gamma$}. When $\gamma=\beta$, we simply write $D_\beta$ for the ample divisor $D_\beta^{\sigma_\beta}$. We keep our assumption that $\beta$ is $E^\vee$-regular. Then Corollary \ref{C:iso} tells us that two moduli spaces $\Mod_\alpha^{\sigma_\beta}(Q_E)$ and $\Mod_{\alpha^\epsilon}^{\sigma_{\beta^\epsilon}}(Q)$ are isomorphic, and it is clear from Lemma \ref{L:proj_ss} that

\begin{lemma} The ample divisors $D_\beta$ and $D_{\beta^\epsilon}$ are equivalent.
\end{lemma}

Remember that $(\beta_\epsilon^\vee)^\epsilon=\tilde{\beta}_\epsilon^\vee$. Under the above identification, the ample divisor $D_{\beta_\epsilon^\vee}$ can be represented by $D_N^{\sigma_{\tilde{\beta}_\epsilon^\vee}}$
for some general $N\in\Rep_{\tilde{\beta}_\epsilon^\vee}(\Perp E)$.
Since the GIT quotients are normal, the fundamental points of $\varphi_E$ in Theorem \ref{T:birational} has codimension at least $2$ \cite[Lemma V.5.1]{Ha}. So it makes sense to talk about pulling back a divisor.

\begin{lemma} \label{L:pullbackD} $\varphi_E^*(D_{\beta_\epsilon^\vee})=D_{\tilde{\beta}_\epsilon^\vee}^{\sigma_\beta}$ wherever $\varphi_E$ is an isomorphism.
\end{lemma}
\begin{proof} It is enough to show that $\varphi_E^*(D_N^{\sigma_{\tilde{\beta}_\epsilon^\vee}})=D_N^{\sigma_\beta}$ for some $N\in\Rep_{\tilde{\beta}_\epsilon^\vee}(\Perp E)$. 
By the construction of $\varphi_E$, we have the following diagram
$$\xymatrix @C=3.6pc {
\Rep_\alpha(E^\reg) \ar[r]^{\R_E} \ar[d]^{q}  & \Mod_\alpha(Q) \ar[d]^{\tilde{q}_E}\\
\Mod_\alpha^{\sigma_\beta}(Q) \ar[r]^{\varphi_E\qquad} & \Mod_\alpha^{\sigma_{\beta_E^\vee}}(Q) \cong \Mod_{\alpha_\epsilon}^{\sigma_{\beta_E}}(Q_E)
}$$

Since the push-forward $q_*$ is essentially taking invariants, it suffices to work with the reduced part of $C_N$ and show that a representation $M\in C_N$ if and only if $\R_E(M)\in C_N$. We recall that $M\in C_N$ if and only if $\hom_Q(N,M)>0$. So what we need is guaranteed by the adjoint property of $\R_E$.
\end{proof}

Now we suppose that $\mc{E}_\alpha$ contains a $\sigma_\beta$-stable point and thus $\sigma_\beta$ is $\mc{E}_\alpha$-effective. Note that the support of $C_{\tau^{-1} E}$ are precisely $\Rep_\alpha(Q)\setminus \Perp E$. By Corollary \ref{C:exc}, $\overline{\mc{E}}_\alpha$ has a scheme structure inherited from $C_{\tau^{-1} E}$, and $E_\beta:=q_*(\overline{\mc{E}_\alpha^{\sigma_\beta}})$ is an irreducible divisor. It follows immediately from the definition of $D_{\tau^{-1} E}^{\sigma_\beta}$ that if $D_{\tau^{-1} E}^{\sigma_\beta}$ is irreducible, then $D_{\tau^{-1} E}^{\sigma_\beta}=E_\beta$.

\begin{lemma} \label{L:Ebeta} In the situation of Theorem \ref{T:blow-up}, $D_{\tau^{-1} E}^{\sigma_\beta}$ is irreducible. So it is the exceptional divisor of the blow-up.
\end{lemma}
\begin{proof} We show that if $C_{\tau^{-1} E}$ has an irreducible components other than $C_{\tau^{-1}\epsilon}$, then those components contain entirely $\sigma_\beta$-unstable points. Suppose that there is a $\sigma_\beta$-stable representation $M$ in other components, then by Lemma \ref{L:frank} the general rank from $M$ to $E$ is not $\epsilon$, say $\gamma$, so $\innerprod{\beta,\gamma}> 0$ and $\innerprod{\tau^{-1}\epsilon,\gamma}\geqslant 0$. But $\beta_\epsilon^\vee\perp\epsilon$, so $\innerprod{\beta+\innerprod{\beta,\epsilon}\tau^{-1}\epsilon,\gamma}\leqslant 0$, and thus $\innerprod{\beta,\gamma}\leqslant0$. A contradiction.
\end{proof}

\begin{theorem} \label {T:IAD}
Suppose that $D_{\tau^{-1} E}^{\sigma_\beta}$ is irreducible and $\hom_Q(\tilde{\beta}_\epsilon^\vee,\alpha-\epsilon)=0$, then $D_\beta=\varphi_E^*(D_{\beta_\epsilon^\vee})-\innerprod{\beta,\epsilon}_Q E_\beta$.
\end{theorem}

\begin{proof} Let $D_{N'}^{\sigma_\beta}$ be a representative for $D_\beta$, and consider $N=N'\oplus \innerprod{\beta,\epsilon}_Q\tau^{-1} E$. The definition of $N$ make sense because $\innerprod{\beta,\epsilon}_Q>0$.
Then by Lemma \ref{L:Dplus} \begin{equation} \label{eq:div} D_{N}^{\sigma_\beta}=D_{N'}^{\sigma_\beta}+\innerprod{\beta,\epsilon}_QD_{\tau^{-1} E}^{\sigma_\beta}=D_\beta+\innerprod{\beta,\epsilon}_QE_\beta.\end{equation}

Next, we claim that for general $N\in\Rep_{\tilde{\beta}_\epsilon^\vee}(Q)$, the support of $D_N^{\sigma_\beta}$ cannot contain that of $E_\beta$ so that their intersection has codimension at most two. This is equivalent to that $C_N$ cannot contain $\Rep_{\alpha\twoheadrightarrow\epsilon}(Q)$. A general element $M$ in the latter has an exact sequence $0\to L\to M\to E\to 0$ with $L$ general in $\Rep_{\alpha-\epsilon}(Q)$. Since $\hom_Q(\tilde{\beta}_\epsilon^\vee,\epsilon)=0$ and $\hom_Q(\tilde{\beta}_\epsilon^\vee,\alpha-\epsilon)=0$, we have that $\hom_Q(N,M)=0$ so $M\notin C_N$.

Since $\varphi_E$ is an isomorphism outside $E_\beta$ and the support of $D_N^{\sigma_\beta}$ intersects $E_\beta$ in codimension at most two, we can apply Lemma \ref{L:Dequi} and \ref{L:pullbackD} and conclude that $\varphi_E^*(D_{\beta_\epsilon^\vee})=D_N^{\sigma_\beta}$.
Then our result $D_\beta=\varphi_E^*(D_{\beta_\epsilon^\vee})-\innerprod{\beta,\epsilon}_Q E_\beta$ follows from the equation \eqref{eq:div}.

\end{proof}

\begin{example} Let us consider quiver $B^1$ with $\alpha=(1,1,1)$ and $\sigma_\beta=(1,1,-2)$, then $\tilde{\beta}_\epsilon^\vee=(2,2,4)$ and $\sigma_{\beta_\epsilon^\vee}=(2,-2)$. We knew that $\varphi_E:\Mod_\alpha^{\sigma_\beta}(B^1)\to\Mod_{\alpha_\epsilon}^{\sigma_{\beta_\epsilon^\vee}}(\Theta_3)$ is the blow-up of $\mb{P}^2$ at one point, and $D_{\beta_\epsilon^\vee}$ clearly corresponds to $\mc{O}(2)$ on $\mb{P}^2$. By Theorem \ref{T:IAD}, $D_\beta=\varphi_E^*(D_{\beta_\epsilon^\vee})-\innerprod{\beta,\epsilon}_Q E_\beta=\varphi_E^*(\mc{O}(2))-E_\beta$. So its linear series corresponds to $|\mc{O}(2)|$ with one assigned base point on $\mb{P}^2$. In particular, $\beta\circ\alpha:=h^0(\Mod_\alpha^{\sigma_\beta}(B^1),D_\beta)=6-1=5$.
\end{example}

\begin{example} Let us consider quiver $C_3^6$ with $\alpha=(3,4,1)$ and $\sigma_\beta=(5,-3,-3)$, then $\tilde{\beta}_\epsilon^\vee=(9,12,9)$ and $\sigma_{\beta_\epsilon^\vee}=(3,-3)$. We knew that $\varphi_E:\Mod_\alpha^{\sigma_\beta}(C_3^6)\to\Mod_{\alpha_\epsilon}^{\sigma_{\beta_\epsilon^\vee}}(\Theta_4)$ is the blow-up of $\mb{P}^3$ along a twisted cubic, and $D_{\beta_\epsilon^\vee}$ clearly corresponds to $\mc{O}(3)$ on $\mb{P}^3$. By Theorem \ref{T:IAD}, $D_\beta=\varphi_E^*(D_{\beta_\epsilon^\vee})-\innerprod{\beta,\epsilon}_Q E_\beta=\varphi_E^*(\mc{O}(3))-E_\beta$. So its linear series corresponds to $|\mc{O}(3)|$ with assigned base points a twisted cubic. The cubic can be defined by the ideal $(x^2-wy,y^2-xz,zw-xy)$, then one can verify that $\beta\circ\alpha:=h^0(\Mod_\alpha^{\sigma_\beta}(C_3^6),D_\beta)=10$.
\end{example}

Among all the weights, the most interesting one is the one related to the anti-canonical character of a representation \cite{Hal}.

\begin{definition}
The {\em anti-canonical character} $\sigma_{ac}$ of the representation $V$ is the character of the representation $\det V$.
\end{definition}

Let us compute the anti-canonical character in the quiver setting, that is, $G=\GL_\alpha$ acts on $V=\Rep_\alpha(Q)$ by the base change.
\begin{align*}
\sigma_{\text{ac}}& =\sum_{a\in Q_1}\alpha(ta)\det{_{ha}}-\alpha(ha)\det{_{ta}}\\
& =\sum_{v\in Q_0}(\sum_{a\in h^{-1}v}\alpha(ta)-\sum_{a\in t^{-1}v}\alpha(ha))\det{_v}.
\end{align*}
So it coincides with the weight $\innerprod{\alpha,-}_Q-\innerprod{-,\alpha}_Q=\innerprod{\alpha+\tau^{-1}\alpha,-}_Q=-\innerprod{-,\alpha+\tau\alpha}_Q$.
We will show under some technical condition $\circledast$, which we think maybe unnecessary in general, that the induced divisor from the anti-canonical weight coincide with the anti-canonical class.
Let $S=\Rep_\alpha\ss{\beta}{ss}(Q)\setminus\Rep_\alpha\ss{\beta}{st}(Q)$ be the strict semi-stable points, the condition $\circledast$ requires $\codim(q(S),\Mod_\alpha^{\sigma_\beta}(Q))\geqslant 2$.

For any $\GL_\alpha$-module $W$, we denote by $\mc{W}$ the sheaf over $\Mod_\alpha^{\sigma_\beta}(Q)$ associated to the module of covariants $(k[\Rep_\alpha(Q)]\otimes W^*)^{\GL_\alpha^{\sigma_\beta}}$. By Luna's slice theorem \cite{Lu}, it is a vector bundle over the geometric quotient $\mc{M}^{\rm st}:=q(\Rep_\alpha\ss{\beta}{st}(Q))$ (at least in the \'{e}tale topology).

\begin{proposition} Under assumption $\circledast$, the induced divisor $D_{ac}^{\sigma_\beta}$ from the anti-canonical weight is the anti-canonical class on $\Mod_\alpha^{\sigma_\beta}(Q)$. In particular, $\Mod_\alpha^{\sigma_{ac}}(Q)$ is a Fano variety.
\end{proposition}

\begin{proof} Let $\hat{G}=\GL_\alpha^{\sigma_\beta}/k^*$, where $k^*$ is embedded multi-diagonally. Note that $\hat{G}$ acts faithfully on $\Rep_\alpha\ss{\beta}{st}(Q)$ and $\mc{M}^{\rm st}$ is smooth. There is a generalized Euler sequence on $\mc{M}^{\rm st}$:
$$0\to\mc{O}_{\mc{M}^{\rm st}}^{\dim\hat{G}}\to\mc{R}ep_\alpha(Q)\to \mc{T}\to 0,$$
where $\mc{R}ep_\alpha(Q)$ is the vector bundle corresponding to the $\GL_\alpha$-module $\Rep_\alpha(Q)$ and $\mc{T}$ is the tangent bundle on $\mc{M}^{\rm st}$. Taking the exterior power to the sequence, we see that the statement holds on $\mc{M}^{\rm st}$. But the assumption $\circledast$ says the complement of $\mc{M}^{\rm st}$ has codimension greater than one, so the two classes agree.
\end{proof}

In view of Lemma \ref{L:acproj}, it is interesting to compute $\innerprod{ac,\epsilon}_Q$. This is equal to $\innerprod{\alpha+\tau^{-1}\alpha,\epsilon}_Q=-\innerprod{\epsilon,\alpha}_Q=-\innerprod{\epsilon,\alpha-\epsilon}_Q-1$, which is one less than the codimension of the blow-up locus in Theorem \ref{T:blow-up}. Of course, this agrees with the general blow-up formula for the canonical divisor \cite[Exercise II.8.5]{Ha}.

\begin{example} (Example \ref{Ex:1c} continued) We now can finish the discussion of this example. Remember that all the weights that we took are anti-canonical, so $h^0(\Mod_\alpha^{\sigma_\beta}(B^n),D_{ac})=10-n$. More generally, the Hilbert polynomial of $D_{ac}$ on $\Mod_\alpha^{\sigma_\beta}(B^n)$ is $\frac{1}{2}[(9-n)x^2+(9-n)x+2]$.
\end{example}

\section*{Acknowledgement}
The author would like to thank Professor Harm Derksen for carefully reading the manuscript, for his great patience, and especially his support for very long time. He also wants to thank Professor Michael Thaddeus and Mihai Halic for answering questions concerning their paper. The author is grateful to Professor Lutz Hille for several good advice.

\bibliographystyle{amsplain}

\end{document}